\documentclass[12pt,leqno]{amsart}
\usepackage{amsmath,stmaryrd}
\usepackage{amssymb}
\usepackage{amsthm}
\usepackage{epsf}
\usepackage{graphicx}
\usepackage{esint} 
\usepackage{dsfont}
\usepackage[pagebackref]{hyperref} 
\hypersetup{pdfpagemode=FullScreen,  colorlinks=true} 
\usepackage{enumerate} 
\usepackage{xcolor,bbm,easybmat,bm}

\usepackage{amsfonts}
\usepackage{amsrefs}
\usepackage{verbatim}
\usepackage{booktabs}
\usepackage{color}
\usepackage{graphicx,float}

\numberwithin{equation}{section}

\newcommand\footnoteref[1]{\protected@xdef\@thefnmark{\ref{#1}}\@footnotemark}
\makeatother

\newcommand{\dr}{\partial}

\DeclareMathOperator{\divg}{div}

\DeclareMathOperator{\loc}{loc}

\newcommand{\BMO}{{\rm BMO}}
\newcommand{\HS}{H\hspace{-0.6mm}S}

\newcommand{\R}{\mathbb R}
\newcommand{\N}{\mathbb N}

\newcommand{\wt}{\widetilde}

\newcommand{\D}{\mathbb D}

\newcommand{\LL}{\mathcal H}
\renewcommand{\L}{L}

\newcommand{\C}{\mathcal C}

\newcommand{\OO}{\mathcal O}

\newcommand{\abs}[1]{\left\vert#1\right\vert}
\newcommand{\br}[1]{\left(#1\right)}
\newcommand{\set}[1]{\left\{#1\right\}}
\renewcommand{\divg}{\mathrm{div}}
\renewcommand{\d}{\, \mathrm{d}} 

\newcommand{\om}{\Omega}
\newcommand{\pom}{\partial\Omega}

\newcommand{\E}{\mathsf{E}} 
\newcommand{\HT}{H_t} 
\newcommand{\dhalf}{D_t^{1/2}} 
\newcommand{\Hdot}{\dot{H}\protect{\vphantom{H}}} 

\newcommand{\pd}{\partial}
\newcommand{\cl}[1]{\overline{#1}} 

\newcommand{\dint}{\int\!\!\!\!\!\int}
\def\Yint#1{\mathchoice
	{\YYint\displaystyle\textstyle{#1}}%
	{\YYint\textstyle\scriptstyle{#1}}%
	{\YYint\scriptstyle\scriptscriptstyle{#1}}%
	{\YYint\scriptscriptstyle\scriptscriptstyle{#1}}%
	\!\dint}
\def\YYint#1#2#3{{\setbox0=\hbox{$#1{#2#3}{\iint}$}
		\vcenter{\hbox{$#2#3$}}\kern-.51\wd0}}
\def\longdash{\mkern-1.5mu{-}\mkern-7.5mu{-}} 

\def\fiint{\Yint\longdash}

\newcommand{\Z}{{\mathbb Z}}

\theoremstyle{plain}
\newtheorem{theorem}[equation]{Theorem}
\newtheorem{lemma}[equation]{Lemma}
\newtheorem{corollary}[equation]{Corollary}
\newtheorem{proposition}[equation]{Proposition}

\newtheorem{definition}[equation]{Definition}

\theoremstyle{definition}

\theoremstyle{remark}

\begin{document}

\title[The $L^p$ Regularity problem for parabolic operators]{The $L^p$ regularity problem for parabolic operators with transversally  independent coefficients}

\author[Dindo\v{s}]{Martin Dindo\v{s}}
\address{School of Mathematics, 
The University of Edinburgh and Maxwell Institute of Mathematical Sciences, Edinburgh, UK}
\email{M.Dindos@ed.ac.uk}

\author[Pipher]{Jill Pipher}
\address{Department of Mathematics, 
 Brown University, RI, US}
\email{jill\_pipher@brown.edu}

\author[Ulmer]{Martin Ulmer}
\address{Department of Mathematics, 
 Brown University, RI, US}
\email{martin\_ulmer@brown.edu}


\begin{abstract} 
In this paper, we fully resolve the question of whether the Regularity problem for the parabolic PDE
$\partial_tu - \divg(A\nabla u)=0$ on the domain $\R^{n+1}_+\times\R$ is solvable for some $p\in (1,\infty)$ under the assumption that the matrix $A$ is elliptic, has bounded and measurable coefficients and its coefficients are independent of the spatial variable $x_{n+1}$ (which is transversal to the boundary). 
We prove that for some $p_0>1$ the Regularity problem is solvable in the range $(1,p_0)$. An analogous 
result for the Dirichlet problem has been considered earlier by Auscher, Egert and Nystr\"om, however the Regularity problem represents an additional step up in difficulty. 
In the elliptic case, the analog of the question considered here was resolved for both the Dirichlet and the Regularity problems by Hofmann, Kenig, Mayboroda, and Pipher.

The main result of this paper complements a recent work of two of the authors with L. Li showing solvability of the parabolic Regularity problem for data in some $L^p$ spaces when the coefficients satisfy
a natural Carleson condition (which is a parabolic analog of the so-called DKP-condition).

\end{abstract}

\maketitle



\tableofcontents

\section{Introduction}


In this paper, we solve the Regularity boundary value problem 
for a class of parabolic operators of 
the form 
\begin{equation}\label{E:pde}
			\LL u:= \dr_t u -\divg (A \nabla u)=0   \quad\text{in } \Omega, 
\end{equation}
defined in a domain $\Omega = \mathcal O \times \R$ where $\mathcal O = \R^{n+1}_+$ or
more generally 
 $\mathcal O$ is an unbounded Lipschitz domain given by a graph $\{x_{n+1}>\phi(x_1,\dots,x_n)\}$ for some Lipschitz function $\phi$.

The matrix $A= [a_{ij}(X, t)]$ is a $n+1\times n+1$ matrix satisfying the uniform ellipticity condition with $X \in \mathcal O$, $t\in \R$.
That is, there exist positive constants $\lambda$ and $\Lambda$ such that
\begin{equation}
	\label{E:elliptic}
	\lambda |\xi|^2 \leq \sum_{i,j} a_{ij}(X,t) \xi_i \xi_j,\qquad  \|A\|_{L^\infty}\leq \Lambda ,
\end{equation}
for almost every $(X,t) \in \Omega$ and all $\xi \in \R^{n+1}$.

  We aim to solve boundary value problems with boundary data and/or certain derivatives of data in some $L^p$ space, for $1 < p < \infty$.
This inquiry has classical
roots in the study of harmonic functions or solutions to the heat equation in smooth domains, where the objective is to prove non-tangential convergence of solutions to
this non-smooth boundary data. (See Definitions~\ref{DefRpar} and~\ref{DefDir}.)
Solving these boundary value problems requires some
additional conditions on the coefficients of the matrix $A$, beyond bounded measurable and elliptic. One natural specific condition on $A$ is the focus of this paper.

The Dirichlet problem involves boundary data prescribed to be
in some $L^p$ space on the boundary, with respect to surface measure, while the Regularity problem imposes an additional smoothness assumption on the Dirichlet data; namely, the existence of ``tangential derivatives" in $L^p$. In addition, solvability entails the more stringent requirement of showing nontangential convergence of the {\it gradient} of solutions to the boundary data, all of which must be suitably defined.

When we study these boundary value problems, we start with the so-called weak solutions that come from 
functional analysis methods.
The classical elliptic $L^p$ Dirichlet problem when $p=2$ can be thought of as being $1/2$ derivative below the natural class of {\it energy solutions}, the weak solutions that arise from applying the Lax-Milgram lemma. These Lax-Milgram solutions have finite energy, meaning that $\iint_{\Omega}|\nabla u|^2<\infty$, and also have traces in the Besov-Sobolev space $\dot {B}^{2,2}_{1/2}(\pom)$. This trace space is an $L^2$-based space of functions having $1/2$-derivative on $\pom$,
appropriately understood. An analogous class of energy solutions that accounts 
for the time derivative can be defined in the parabolic setting, and this is done in subsection \ref{RwEs}. 
Hence we will see that the $L^p$ Regularity problem when $p=2$ can be thought of as being $1/2$ derivative above the energy class of solutions discussed above.\vglue1mm

Solving the Dirichlet problem is ``easier" than solving Regularity for a number of reasons. First,
solutions to operators like \eqref{E:pde}, even when the coefficients of the elliptic matrix $A$ are merely bounded and measurable, are known to be H\"older continuous of some order (\cite{Na}). But the derivatives of solutions are not continuous, and are not necessarily defined pointwise. This brings in new difficulties even in formulating the nontangential convergence problem. 
Second, the study of the Dirichlet problem can essentially be understood as a study of the {\it elliptic/parabolic measure}. Properties of this measure, including tools that involve positivity of solutions and maximum principles, are useful for studying the Dirichlet problem, but insufficiently useful when studying Regularity. For these latter investigations, a completely different approach is needed, requiring a direct bound of the nontangential maximal function. Finally, the interpolation and localization needed to extrapolate from solvability in one $L^p$ space to solvability in a range of $L^q$ spaces is simply more complicated for the Regularity and Neumann problems, by comparison to Dirichlet. 

Motivation for the study of the Regularity problem comes
 from several directions. In some sense, it is a companion problem - and, generally speaking, an easier one - to the Neumann problem, which prescribes the normal, or co-normal, derivative at the boundary. In several situations, solvability of 
 the Regularity problem has preceded and aided in solving the 
 Neumann problem. However, it turns out there are more direct relationships between the Regularity problem for an operator, and the Dirichlet problem for its adjoint. In the setting of elliptic divergence form operators in Lipschitz domains, a certain duality was observed (\cite{KP1}) for the Regularity problem. Namely, the solvability of the Regularity problem with data in $L^p$ implies solvability of the Dirichlet problem for the adjoint operator in \(L^{p'}\), where $p'$ and $p$ are dual 
 exponents ($1/p + 1/p' = 1$).  This same duality was extended to the present setting of time-varying parabolic operators (and also in time-varying Lipschitz cylinders) in \cite{DinD}. 
So far, this duality in full generality is only known to go in one direction. However, with additional assumptions on the coefficients  and on the domain, there are positive results (\cite{DHP}, \cite{AM} and \cite{MPT}), while in rougher (fractal) domains, there are negative results (\cite{Ga}).

Before setting the stage for the parabolic results contained here, we give a high-level
overview of the current state of knowledge on boundary value problems  in the elliptic setting, that is, for operators
of the form $L = -\divg (A \nabla)$.
There are relatively
few, primarily two\footnote{Some conditions that do not quite fit into this framework, and for which solvability of boundary value problems have been studied, include the recent work \cite{U2} and \cite{DGQM}.}, 
natural conditions on the coefficients that have been well studied in the elliptic setting. Note, that these two conditions include many previously considered cases such as constant/smooth/$C^\alpha$ or Dini condition on modulus of continuity of coefficients, which we therefore do not mention further. 

The first condition, introduced in \cite{KPCarl}, requires that a certain expression in the gradient (and, in the parabolic setting, the time derivative) of the coefficients defines a Carleson measure on the domain. It arises naturally from change-of-variable considerations and has a long and rich history that we cannot do justice to here. Suffice it to say that a number of different boundary value problems---Dirichlet, Regularity, and Neumann---have been studied, and solved, under various assumptions such as the size of the Carleson measure or the geometry of the domain. Papers \cite{DHP}, \cite{DPR}, and \cite{MPT} are especially relevant to the results here, complex coefficients have been studied in \cite{DP1}, \cite{DP2}, and \cite{hofmann_layer_2015}, and many more references can be found in the survey article \cite{DP-vietnam}.

A second, conceptually different line of inquiry has developed in parallel concerning solvability of
the elliptic Dirichlet and Regularity problems under a second and different condition that is meaningful under certain assumptions on the domain (such as Lipschitz boundaries). Here, one imposes the condition that the coefficients of the matrix $A$ are independent of the variable transversal to the boundary. In
the setting of the upper half space, and with our notation, this would be stated as:
\begin{equation}\label{Cone-ell}
A(x,x_{n+1}) = A(x) \qquad\mbox{for all }x=(x_1,x_2,\dots,x_n)\in\R^n\mbox{ and }x_{n+1}>0.
\end{equation}
Once again, this condition of independence of the transversal variable arose naturally from a change of variables mapping the domain above a Lipschitz graph to the upper half space. For symmetric matrices satisfying \eqref{Cone-ell}, \cite{JK} showed solvability of the Dirichlet problem for data in $L^2$, obtaining a generalization of Dahlberg's celebrated result for the Laplacian on Lipschitz domains. When the matrix is not necessarily symmetric, it turns out that the best one can say is that the Dirichlet problem is solvable for data in {\it some} $L^p$ space (\cite{KKPT}, \cite{HKMP1}), and the Regularity problem is solvable in the dual range (\cite{HKMP2}). The results of \cite{HKMP1} and \cite{HKMP2} relied crucially on the technology developed to solve the Kato conjecture (\cite{AHLMT}) for complex valued elliptic operators
(with matrices in block form and satisfying the 
above independence condition).

The parabolic setting always presents
additional obstacles, and therefore the development of the theory often lags its elliptic counterpart. Recently, the 
paper \cite{DLP} closed one gap, solving the $L^p$ Regularity problem for the PDE $\LL u=0$ for a certain range of $p$ assuming (1) that
the measure defined by
\begin{equation}\label{E:1:carl}
d\mu = \left( \delta(X)|\nabla A|^2 + \delta(X)^3|\partial_t A|^2  \right) dX dt
\end{equation}
is the density of a Carleson measure on $\Omega$ with Carleson norm $\|\mu\|_C$, and (2) that
\begin{equation}\label{E:1:bound}
\delta(X)|\nabla A| + \delta(X)^2|\partial_t A| \leq K<\infty.
\end{equation}
Here and in the sequel, $\delta(X)$ denotes the parabolic distance from $X$ to the boundary of $\om$. Since $\partial\Omega=\partial\mathcal O\times\R$, the infimum
\[
\inf_{(Y,\tau)\in\pom}\br{\abs{X-Y}^2+\abs{t-\tau}}^{1/2}
\]
is attained at $\tau=t$, so it depends only on $X$ and reduces to the Euclidean distance
\[
\delta(X)=\inf_{Y\in\partial\mathcal O}\abs{X-Y}.
\]

This left open the question of solvability of the parabolic Regularity problem for operators that satisfy the appropriate 
analog of \eqref{Cone-ell}, and is our focus here.
For the PDE $\LL u=0$ and for $\mathcal O = \mathbb R_+^{n+1}$, we assume the following independence condition: 
\begin{equation}\label{Cond-par}
A(x,x_{n+1},t) = A(x,t) \qquad\mbox{for all }x=(x_1,x_2,\dots,x_n)\in\R^n,\, x_{n+1}>0\mbox{ and }t\in\R,
\end{equation}
as we allow the matrix $A$ to depend on the time variable as well. 
(When convenient we will also refer to the variable $x_{n+1}$ as $\lambda$ in order to shorten notation.)

It is important to emphasize that conditions \eqref{E:1:carl}-\eqref{E:1:bound} and \eqref{Cond-par}
are logically independent. Neither condition implies the other, and together with the recent
work in \cite{DLP}, the present paper completes the picture for the parabolic Regularity
problem under the two principal frameworks currently studied in the literature.

Under the condition \eqref{Cond-par} the solvability of the  $L^p$ Dirichlet problem for the equation \eqref{E:pde} was 
resolved in \cite{AEN2}. The parabolic Kato conjecture was solved in \cite{AEN}, and, using a different approach, the parabolic Kato conjecture with weights was resolved in \cite{AtEN}. The 
perspective and methodology of \cite{AtEN} proved particularly relevant and valuable for the results of this paper.

We see that it is natural and interesting to solve this Regularity problem for general time-varying parabolic operators whose 
coefficients satisfy  \eqref{Cond-par}
in the optimal dual $L^p$ range, especially as it completes the picture in the parabolic case for both Dirichlet and Regularity problems under the two well
studied matrix conditions: \eqref{E:1:carl} or \eqref{Cond-par}.
To this end we show solvability of the Regularity problem in an interval $(1,p_0)$
for some $p_0>1$, and then argue that, moreover, $p_0$ is determined by the range 
$(p_0',\infty)$, of solvability of the Dirichlet problem for the adjoint PDE: $-\partial_tu-\divg(A^*\nabla u)=0$,   where $1/p_0+1/p_0'=1$.
 \vglue1mm



We conclude this introduction by pointing out another reason that,
historically, the development of the parabolic theory has lagged that of the elliptic theory. In the parabolic setting, the time variable enjoys a different scaling from that of the spatial variables: the equation 
tells us that, roughly speaking,  $\partial_t u\sim \nabla^2 u$. That is, one spatial gradient of $u$ corresponds to $1/2$ derivative in time.
Hence we find that the class of energy solutions (c.f. \cite{AEN}) for parabolic PDEs is the space with {\it finite energy}:
$$\iint_{\Omega}|\nabla u|^2+\iint_{\Omega}|D_t^{1/2}u|^2<\infty.$$
The presence of the half derivative in time, a non-local operator, is responsible for many additional difficulties encountered in this setting. 
The finite energy solutions have traces on the boundary with $1/2$-spatial and $1/4$-time derivatives;  these considerations are explained with care in subsection 2.2 of \cite{DLP} and apply without changes to the setting of this paper.
Therefore, in formulating the parabolic $L^p$-based Regularity problem, one looks for solutions of the parabolic equation with prescribed data having one full spatial derivative and a half-time derivative on $\pom$. Some of the difficulties in dealing with the non-local half derivative have been mitigated by the results of \cite{Din23}, but there are still challenges in using the equation to recover half derivatives on the boundary. 

For some further context, we note that
the initial formulation of the Regularity problem for the heat equation is from \cite{FR}, where it was formulated for $C^1$ cylinders. This was then re-formulated and solved in
bounded Lipschitz cylinders in \cite{Bro87, Bro89}.  Constant coefficient systems were studied in \cite{Nys06}.
The parabolic Regularity problem with symmetric, time-independent variable coefficients for H\"older continuous coefficients in bounded Lipschitz cylinders has been solved in \cite{CRS} for the $p=2$ case. A similar result for systems (with smooth, symmetric and  time-independent coefficients) is done in \cite{M} in $L^p$ for $p$ close to $2$. 

\subsection{Main Results and Methodology}

Our main result is the following theorem that optimally resolves the range of solvability of parabolic PDE with coefficients satisfying \eqref{E:elliptic} and \eqref{Cond-par}.

At the level of strategy, the proof follows the roadmap of \cite{HKMP2}, which solved the Regularity problem in the elliptic setting: it proceeds through explicit auxiliary constructions, integrations by parts, and a detailed decomposition argument. The main new difficulties arise when we need to establish $L^p$ bounds of certain key operators that now involve the time variable; for some of these the $L^2$ bounds are known thanks to \cite{AEN2}, but the extension to $L^p$ for $p\ne 2$ is non-trivial. To handle these, and motivated by the technique of \cite{DLP}, we estimate the Regularity problem indirectly by pairing against solutions of an adjoint inhomogeneous problem $-\dr_t u -\divg (A^* \nabla u)=\divg\,\vec{h}$ with vanishing Dirichlet data, and then reduce matters to square-function and nontangential maximal-function estimates inherited from the Dirichlet theory for the adjoint operator. This auxiliary mechanism, which is what underlies Section~\ref{Formulation} and the reduction in Section~\ref{SA}, is in the spirit of the Poisson-duality framework developed most fully in the elliptic setting in \cite{MPT} (and used previously in \cite{KP2, DHP}, and in parabolic settings in \cite{DLP} for an explicitly constructed class of test functions).

We now briefly address our assumption that the domain of the boundary value problem under consideration has the form $\Omega = \mathcal O \times \R$, i.e., the domain does not vary in time. In the present setting, the standard approach via a change of variables $(x,x_{n+1},t)\mapsto (x,x_{n+1}-\lambda(x,t),t)$ for some function $\lambda=\lambda(x,t)$ would destroy the structure of our PDE \eqref{E:pde} for the $\partial_t u$ term. 
Likewise, in the paper \cite{DLP} with the imposed conditions \eqref{E:1:carl}-\eqref{E:1:bound}, time-varying domains were not considered, again because the approach via a change of variables changed the structure of the equation. The (different) change of variables that is compatible with the Carleson condition  leads to a new PDE with a first order term (drift). Even in 
the elliptic case, we are not able to handle solvability of the Regularity problem with drift terms. Hence for now only time-independent domains are considered.
\medskip

\begin{theorem}\label{MainT} Assume that the parabolic PDE $\LL u= \dr_t u -\divg (A \nabla u)=0$ is defined on the domain  $\Omega=\mathcal O\times\R$, where $\mathcal O\subset\R^{n+1}$ is an unbounded Lipschitz domain
of the form:
$$\mathcal O=\{(x,x_{n+1}): x\in\R^n\mbox{ and }x_{n+1}>\phi(x)\},\quad\mbox{for some Lipschitz function $\phi$},$$
and that the matrix $A$ is uniformly elliptic \eqref{E:elliptic} with bounded and measurable coefficients. Assume, in addition, that
\begin{equation}\label{Cond-par2}
A(x,x_{n+1},t) = A(x,t) \qquad\mbox{for all }x=(x_1,x_2,\dots,x_n)\in\R^n,\,x_{n+1}>\phi(x)\mbox{ and }t\in\R.
\end{equation}
Then there exists $p_0>1$ such that for all $1<p<p_0$ the $L^p$ Regularity problem for the equation $\LL u=0$ in $\Omega$ is solvable (c.f. Definition \ref{DefRpar}). Moreover, this interval of solvability $(1,p_0)$ is dual to the interval of solvability of the $L^q$ Dirichlet problem (c.f. Definition \ref{DefDir}) for the adjoint equation 
\begin{equation}\label{E:adjpde}
     \LL^*u:=-\dr_tu-\divg(A^*\nabla u)=0
\end{equation} in $\Omega$ which is solvable for all $q>p_0'$.
\end{theorem}
We note that the duality between the Regularity problem for $\LL$ in the interval $(1,p_0)$ and the Dirichlet problem for $\LL^*$ in the interval $(p'_0,\infty)$ is sharp, i.e., $p_0$ and $p_0'$ are natural endpoints at which solvability fails; this is a consequence of the papers \cite{DinD, DinN}, where the first paper establishes the direction \lq\lq Regularity for $\LL$ implies Dirichlet for $\LL^*$" and the latter paper the reverse direction. In particular, by analogy with the elliptic theory of \cite{HKMP2, AM}, one cannot in general expect solvability of the Regularity problem beyond the interval $(1,p_0)$.
The other endpoint for the Regularity problem at $p=1$ holds with boundary data in a certain Hardy-Sobolev space $\dot{\HS}^{1}_{1,1/2}(\R^n\times\R)$ which we introduce in Definition \ref{HSato}. This follows again
from the result in \cite{DinN}. 
\medskip

\noindent {\bf Acknowledgement:} We would like to thank Steve Hofmann for some helpful discussions, especially for pointing us to the paper \cite{AtEN}. We also thank the anonymous referee for a careful reading of our paper and for very helpful suggestions to improve its readability.

\section{Formulation of the Regularity problem}\label{Formulation}
\setcounter{equation}{0}

\subsection{Parabolic Sobolev Spaces on {$\partial\Omega$}}%
\label{S:paraSobolev}

The appropriate function spaces for our boundary data should reflect the same homogeneity as the PDE\@.
As a rule of thumb, one derivative in time behaves like two derivatives in space and so the correct order of our time derivative should be $1/2$ when the data is assigned to have one derivative in spatial variables.
This has been studied previously in~\cites{HL96,HL99,Nys06}, who have followed~\cite{FJ68} in defining the homogeneous parabolic Sobolev space $\dot{L}^p_{1,1/2}$ in the following way. 

\begin{definition}%
	\label{D:paraSob}
	The \textit{homogeneous parabolic Sobolev space} $\dot{L}^p_{1,1/2}(\R^n\times\R)$, for $1 < p < \infty$, is defined to consist of equivalence classes of functions $f$ with distributional derivatives satisfying $\|f\|_{\dot{L}^p_{1,1/2}(\R^n\times\R)} < \infty$, where
	\begin{equation}
		\|f\|_{\dot{L}^p_{1,1/2}(\R^n\times\R)} = \|\D f\|_{L^p(\R^n\times\R)}
	\end{equation}
	and
	\begin{equation}\label{eq2.4}
		(\D f)\,\widehat{\,}\,(\xi,\tau) := \|(\xi,\tau)\| \widehat{f}(\xi,\tau).
	\end{equation}
Here $\|(\xi,\tau)\|$ on $\R^{n} \times \R$ is defined as the unique positive solution $\rho$ to the following equation
\begin{equation}
	\label{E:par-norm}
	\frac{|\xi|^2}{\rho^2} + \frac{\tau^2}{\rho^4} = 1.
\end{equation}
It is the case that $\|(\xi,\tau)\| \sim (|\xi|^2 + |\tau|)^{1/2}$ so that this norm has the correct parabolic scaling. 
\end{definition}

In addition, following~\cite{FR67}, we define a parabolic half-order time derivative by
\begin{equation}
	(\D_{n+1} f)\,\widehat{\,}\,(\xi,\tau) := \frac{\tau}{\|(\xi,\tau) \|} \widehat{f}(\xi,\tau).
\end{equation}

If $0 < \alpha \leq 2$, then for $g \in C^\infty_c(\R)$ the \textit{one-dimensional fractional differentiation operators} $D_\alpha$ are defined by
\begin{equation}
	(D^\alpha g)\,\widehat{\,}\,(\tau) := |\tau|^\alpha \widehat{g}(\tau).
\end{equation}
It is also well known that if $0 < \alpha < 1$ then
\begin{equation}
	D^\alpha g(s) = c\int_{\R} \frac{g(s) - g(\tau)}{|s-\tau|^{1 + \alpha}} d\tau
\end{equation}
whenever $s \in \R$.
If $h(x,t) \in C^\infty_c(\R^n\times\R)$ then by $D^\alpha_t h: \R^n\times\R \rightarrow \R$ we mean the function $D^\alpha h(x, \cdot)$ defined a.e. for each fixed $x \in \R^{n}$.
We now recall the established connections between $\D$, $\D_{n+1}$ and $D_t^{1/2}$. From ~\cites{FR66,FR67, DinD} we have that
\begin{align}\label{eqrn}
	\|\D f\|_{L^p(\R^n\times\R)} \sim  \|\D_{n+1} f\|_{L^p(\R^n\times\R)}+ \|\nabla f\|_{L^p(\R^n\times\R)}\sim  \|D_t^{1/2} f\|_{L^p(\R^n\times\R)}+\|\nabla f\|_{L^p(\R^n\times\R)},
\end{align}
for all $1<p<\infty$. Here, $\nabla$ denotes the usual gradient in the variables $x\in\mathbb R^{n}$.\vglue1mm

When $\Omega=\mathcal O\times\R$, for 
$$\mathcal O=\{X=(x,x_{n+1}): x\in\R^n\mbox{ and }x_{n+1}>\phi(x)\},\quad\mbox{for some Lipschitz function $\phi$},$$
we will need an extension of the definition of spaces $\dot L^p_{1,1/2}$. This is rather trivial however, as we consider the projection $\pi:\partial\Omega\to \R^n\times\R$ defined by
$$(x,\phi(x),t)\mapsto (x,t),$$
and then use pull-back to define $\dot L^p_{1,1/2}(\partial\Omega)$ using the fact that $\dot L^p_{1,1/2}(\R^n\times\R)$ has been defined earlier.

\subsection{Energy solutions}\label{RwEs} We summarize the results of \cite{AEN,DLP} where this concept was extensively discussed. We say that a function $v\in \dot{\E}_{\loc}(\mathcal O\times\mathbb R)$
belongs to the \emph{energy class} $\dot \E(\mathcal O\times\mathbb R)$ if
\begin{align*}
 \|v\|_{\dot \E} := \bigg(\|\nabla v\|_{\L^2(\mathcal O\times\mathbb R)}^2 + \|\HT \dhalf v\|_{\L^2(\mathcal O\times\mathbb R)}^2 \bigg)^{1/2} < \infty.
\end{align*}
Here $H_t$ denotes the Hilbert transform in the $t$-variable; recall that $H_t$ is an isometry on $L^2$ and bounded on $L^p$, $1<p<\infty$. It arises from the decomposition $\partial_t= \dhalf H_t  \dhalf$ and commutes with $ \dhalf$ as both are Fourier multiplier operators.\medskip

When considered modulo constants, $\dot \E $ is a Hilbert space and, in fact, is the closure of $\C_0^\infty\!\big(\,\cl{\mathcal O\times\mathbb R}\,\big)$ for the homogeneous norm $\|\cdot\|_{\dot \E}$.  As shown in \cite{AEN} (with a small generalization), functions from $\dot \E $ have well defined traces with values in
 the \emph{homogeneous parabolic Sobolev space} $\Hdot^{1/4}_{\pd_{t} - \Delta_x}(\partial\mathcal O\times\mathbb R)$.  Conversely, any $g \in \Hdot^{1/4}_{\pd_{t} - \Delta_x}$ can be extended to a function $v \in  \dot \E$ with trace $v\big|_{\partial\mathcal O\times\R} = g$.
  
Hence, by the {\it energy solution} to $\partial_tu - \divg(A\nabla u)=0$ with Dirichlet boundary datum $u\big|_{\partial\mathcal O\times\R} = f \in \Hdot^{1/4}_{\pd_{t} - \Delta_x}$ (understood in the trace sense) we mean $u \in \dot\E$ such that
\begin{align*}
 \iint_{\mathcal O\times\R} \left[A \nabla u \cdot{\nabla v} + \HT \dhalf u \cdot {\dhalf v}\right] \d X \d t = 0,
\end{align*}
holds for all $v \in \dot \E_0$, the subspace of $\dot \E$ with zero boundary trace.

\subsection{$(R)_p$ boundary value problem}

 Let $\Omega=\OO\times \mathbb{R}$, where $\OO\subset\R^{n+1}$ is as before a Lipschitz domain. Assume that $A:\Omega\to M_{(n+1)\times(n+1)}(\mathbb R)$ is a bounded uniformly elliptic matrix-valued function. Let $\LL=\dr_t-\divg A\nabla$.
 \vglue1mm

 To give the definition of the parabolic Regularity problem, we introduce notation for parabolic balls and cubes, and then define nontangential parabolic cones and parabolic non-tangential maximal functions.  
 
 \begin{definition}
A parabolic cube on $\R^{n+1}\times\R$  centered at $(X,t)$ with sidelength $r$ is defined as
$$    Q_r(X,t):=\{ (Y, s) \in \R^{n+1}\times\R : |x_i - y_i| < r \ \text{ for } 1 \leq i \leq n+1, \ | t - s |^{1/2} < r \}.
$$
When writing a lower case point $(x,t)$ we shall mean a boundary parabolic cube on $\R^{n}\times\R$
which has an analogous definition but in one less spatial dimension:
\begin{equation}\label{eqdef.bdypcube}
    Q_r(x,t):=\{ (y, s) \in \R^{n}\times\R : |x_i - y_i| < r \ \text{ for } 1 \leq i \leq n, \ | t - s |^{1/2} < r \}.
\end{equation}
A parabolic ball on $\R^{n+1}\times\R$  centered at $(X,t)$ with radius $r$ is the ball
\begin{equation}\label{eqdef.ball}
    B_r(X,t):=\{ (Y, s) \in \R^{n+1}\times\R : \|(X-Y,t-s)\|<r \},
\end{equation}
where $\|\cdot\|$ has been defined earlier by \eqref{E:par-norm}. Recall that $\|\cdot\|$ scales as the parabolic distance function
\[
d_p((X,t),(Y,s)) := \br{\abs{X-Y}^2+\abs{t-s}}^{1/2}\sim \|(X-Y,t-s)\|.\]

For parabolic balls at the boundary we use notation $\Delta_r(X,t)=B_r(X,t)\cap \pom$. In the special case $\Omega=\R^{n+1}_+\times\R$ we drop the last spatial coordinate and also write $\Delta_r(x,t)$ with understanding that the ball is centered at $(X,t)=(x,0,t)$.
\end{definition}

\begin{definition}
For $a>0$ and $(z,\tau)\in\pom$, unless otherwise defined, we denote the non-tangential parabolic cones by
\begin{equation}\label{Gamma2.11}
    \Gamma_a(z,\tau):=\set{(X,t)\in\om: d_p((X,t),(z,\tau))<(1+a)\delta(X)},
\end{equation}

and $\delta(\cdot)$ is the parabolic distance to the boundary, which under the assumption $\Omega=\mathcal O\times\R$ reduces to the Euclidean distance:
\[
\delta(X)=\inf_{Y\in\partial\mathcal O}\abs{X-Y}.\]
\end{definition}

\begin{definition}\label{def-N}
For $w\in L^\infty_{\loc}(\om)$, we define the non-tangential maximal function of $w$ as
\[
 N_a(w)(z,\tau):=\sup_{(X,t)\in\Gamma_a(z,\tau)}\abs{w(X,t)} \quad\text{for }(z,\tau)\in\pom.
\]
If $w\in L^p_{\loc}(\om)$, $p\in(0,\infty)$, we need the modified non-tangential maximal function
\begin{equation}\label{def.Nap}
    \wt N_{a,p}(w)(z,\tau):=\sup_{(X,t)\in\Gamma_a(z,\tau)}\br{\fiint_{B_{\delta(X)/2}(X,t)}\abs{w(Y,s)}^pdYds}^{1/p} \quad\text{for }(z,\tau)\in\pom,
\end{equation}
We drop the use of the subscript $p$ when $p=2$, that is
\begin{equation}\label{def.N2}
    \wt N_{a}(w):=\wt N_{a,2}(w).
\end{equation}
\end{definition}

It is well-known (using a level-sets argument) that for $p\in(0,\infty)$, the $L^p$ norms of $N_a(w)$, $\wt{N}_a(w)$, are invariant under changes of $a$ up to a constant multiple. For this reason, we  omit the dependence on the aperture $a$ of the cones when there is no need for the specificity. 
\medskip

We define two more objects, the area integral and the square function. 

\begin{definition} For $F\in L^2_{loc}(\Omega)$ and $(z,\tau)\in\partial\Omega$ let
$$\mathcal A(F)(z,\tau)=\left(\iint_{\Gamma(z,\tau)}\frac{|F(Y,s)|^2}{\lambda^{n+3}}dyd\lambda ds\right)^{1/2},$$
and for $v\in W^{1,2}_{loc}(\Omega)$ we set $S(v)(z,\tau)=\mathcal A(\lambda\nabla v)(z,\tau)$.
\end{definition}

We also introduce the Carleson measure expressions which arise
when studying area integral/nontangential maximal function estimates:

\begin{definition} For $F\in L^2_{loc}(\Omega)$ and $(z,\tau)\in\partial\Omega$ let
$$C(F)(z,\tau)=\sup_{\Delta=\Delta(z,\tau)\subset\partial\Omega}\frac1{|\Delta|}\iint_{T(\Delta)}\frac{|F(Y,s)|^2}{\lambda}dyd\lambda ds.$$
Here, as is customary, $T(\Delta)$ is the Carleson region associated with the boundary ball $\Delta$,
which for $\Delta=B\cap\partial\Omega$ can be taken to be $T(\Delta):=B\cap\Omega$.
\end{definition}

 We can now formulate our definition of the parabolic Regularity problem $(R)_p$. 
\begin{definition}\label{DefRpar}
Let $1<p<\infty$.
The Regularity problem for the parabolic operator $\LL$ with boundary data in
$\dot{L}^{p}_{1,1/2}(\partial\Omega)$ is solvable (abbreviated $(R)_{p}$), if
for every $f\in \dot{L}^{p}_{1,1/2}(\partial\Omega)\cap \Hdot^{1/4}_{\partial_t-\Delta_x}(\partial\Omega)$ 
the {\rm energy solution} $u\in  \dot \E(\Omega)$ of $\LL u=0$  to the problem
\begin{align}\label{eq-pp}
\begin{cases}
\LL u&
=0 \quad\text{ in } \Omega\\
u|_{\partial \Omega} &= f \quad\text{ on } \partial \Omega
\end{cases}
\end{align}
satisfies
\begin{align}
\label{RPpar} \quad\|\widetilde{N}(\nabla u)\|_{L^p(\partial
\Omega)}\lesssim \|\nabla_T f\|_{L^{p}(\partial\Omega)}+\|D^{1/2}_t f\|_{L^{p}(\partial\Omega)}\approx\| f\|_{\dot{L}_{1,1/2}^{p}(\partial\Omega)}.
\end{align}
The implied constant depends only on the matrix $A$, $p$ and $n$.
\end{definition}

When we need to specify the operator, we will use the abbreviation $(R)_p^{\LL}$.

Observe that this definition does not require control of  $\|\widetilde{N}(D^{1/2}_tu)\|_{L^p}$ or of $\|\widetilde{N}(H_tD^{1/2}_tu)\|_{L^p}$ in 
 \eqref{RPpar} as the following theorem shows that these bounds follow from \eqref{RPpar}.

\begin{theorem}(c.f. \cite{Din23})\label{timp} Let $\Omega$ be an infinite Lipschitz cylinder of the form $\mathcal O\times\mathbb R$, $\LL$ be as above and assume that for some $1<p<\infty$ the Regularity problem $(R)_p$ for $\LL$ on the domain $\Omega$ is solvable. Then in addition to \eqref{RPpar} we also have the bounds
\begin{align}
\label{RPpar2} \quad\|\widetilde{N}(D^{1/2}_t u)\|_{L^p(\partial
\Omega)}+\|\widetilde{N}(H_tD^{1/2}_t u)\|_{L^p(\partial
\Omega)}\lesssim \| f\|_{\dot{L}_{1,1/2}^{p}(\partial\Omega)}.
\end{align}
\end{theorem}

Theorem \ref{timp} shows that our definition coincides with the previous notions in the literature of solvability of the Regularity problem, as for example in \cite{Bro87, Bro89, Nys06, M, CRS} where one or both of the terms on the left-hand side of \eqref{RPpar2} were incorporated into the statement of the problem. This theorem significantly simplifies our task, since we need 
only to focus on bounds for the spatial gradient of a solution $u$.\vglue1mm

For completeness we also state the definition of the $L^p$ Dirichlet problem $(D)_p$:

\begin{definition}\label{DefDir} Let $p\in (1,\infty)$. 
We say that the Dirichlet problem for the operator $\LL$ with boundary data in
${L}^{p}(\partial\Omega)$ is solvable (abbreviated $(D)_{p}$), if
for every $f\in {L}^{p}(\partial\Omega)\cap \Hdot^{1/4}_{\partial_t-\Delta_x}(\partial\Omega)$ 
the {\rm energy solution} $u\in  \dot \E(\Omega)$ (as defined above) to the problem \eqref{eq-pp} satisfies the estimate
\begin{align}
\label{RPdir} \quad\|\widetilde{N}( u)\|_{L^p(\partial
\Omega)}\lesssim \| f\|_{L^{p}(\partial\Omega)}.
\end{align}
The implied constant again depends only the matrix $A$, $p$ and $n$.
\end{definition}

\section{Solvability of the Regularity problem - reduction to $11$ key bounds}
\setcounter{equation}{0}

The following section reduces the question of solvability of the Regularity problem for parabolic PDEs with $x_{n+1}$-independent coefficients to the question of boundedness of certain square functions in $L^p$.
 For simplicity we assume that  $\Omega=\R^{n+1}_+\times\R$ in the whole section.
\medskip

We start by adapting useful results from \cite{KP2} to parabolic settings. Let us denote by $\tilde{N}_{1,\varepsilon}$ the $L^1$-averaged version of the non-tangential maximal function for {\it doubly truncated} parabolic cones. That is, for $\vec{u}:{\mathbb R}^{n+1}_+\times\mathbb R\to\mathbb R^m$, we set
$\Gamma^\varepsilon(Y,s):=\Gamma(Y,s)\cap
\{(X,t): \varepsilon < \delta(X)< 1/\varepsilon\}$, 
and
$$\tilde{N}_{1,\varepsilon}(\vec{u})(Y,s)=\sup_{(X,t)\in\Gamma^\varepsilon(Y,s)}\fiint_{(Z,\tau)\in B_{\delta(X)/4}(X,t)}|\vec{u}(Z,\tau)|dZd\tau.$$

Lemma 2.8 of \cite{KP2}, stated below, provides a way to estimate the $L^q$ norms of $\tilde{N}_{1,\varepsilon}(\nabla F)$ via duality (based on tent-spaces). 

\begin{lemma}\label{l1bb} Let $q>1$. There exist $\vec{\alpha}(X,t,Z,\tau)$ with $\vec{\alpha}(X,t,\cdot):B(X,t,\delta(X)/2)\to{\mathbb R}^{n+2}$ and \newline $\|\vec\alpha(X,t,\cdot)\|_{L^\infty(B(X,t,\delta(X)/2))}=1$, a nonnegative scalar function $\beta(X,t,Y,s)\in L^1(\Gamma^\varepsilon(Y,s))$ with \newline $\iint_{\Gamma^\varepsilon(Y,s)}\beta(X,t,Y,s)\,dXdt=1$ and a nonnegative $g\in L^{q'}(\partial{(\mathbb R^{n+1}_+\times\mathbb R)},d\sigma)$ with $\|g\|_{L^{q'}}=1$ such that
\begin{equation}
\left\|\tilde{N}_{1,\varepsilon}(\nabla F)\right\|_{L^q(\partial{(\mathbb R^{n+1}_+\times \mathbb R)},d\sigma)}\lesssim \iint_{{\mathbb R^{n+1}_+\times \mathbb R}}\nabla F(Z,\tau)\cdot \vec{h}(Z,\tau)\, dZd\tau,
\label{e1a}
\end{equation}
where
$$\vec{h}(Z,\tau)=\int_{\partial{(\mathbb R^{n+1}_+\times \mathbb R)}}\iint_{\Gamma(Y,s)}g(Y,s)\vec{\alpha}(X,t,Z,\tau)\beta(X,t,Y,s)\frac{\chi_{B(X,t,\delta(X)/4)}(Z,\tau)}{\delta(X)^{n+2}}\,dX\,dt\,dY\,ds,$$
and $\chi_A$ is the characteristic function of the set $A$.

Moreover, for any $G:{\mathbb R^{n+1}_+\times\mathbb R}\to\mathbb R$ with $\tilde{N}_{1}(\nabla G)\in L^q(\partial({\mathbb R^{n+1}_+}\times\mathbb R))$ we also have an upper bound
\begin{equation}
\left|\iint_{{\mathbb R^{n+1}_+\times\mathbb R}}\nabla G(Z,\tau)\cdot \vec{h}(Z,\tau)\, dZ\,d\tau\right|\lesssim \left\|\tilde{N}_{1}(\nabla G)\right\|_{L^q(\partial({\mathbb R^{n+1}_+}\times\mathbb R),d\sigma)},
\label{e2}
\end{equation}
where $\wt N_1$ is the modified non-tangential maximal function with $p=1$ in \eqref{def.Nap}.
The implied constants in \eqref{e1a}-\eqref{e2} do not depend on $\varepsilon$, only on the dimension $n$.
\end{lemma}

The approach via duality was also the strategy of \cite{DLP}, where a more sophisticated version of the test function $\vec{h}$ is constructed, one that is also differentiable. In particular, \cite{DLP} shows that the extension of the originally elliptic statement in \cite{KP2} to the parabolic setting works as expected: no further modifications are required beyond replacing elliptic cones, balls, and distances with their parabolic counterparts (with the natural inhomogeneous parabolic scaling). In particular, all relevant geometric covering and tent-space arguments transfer verbatim to the parabolic setting. Later, in section 7, we will make use of 
the fact that $\vec{h}(Z,\tau)$ has a tent space bound, and this bound follows exactly as it does in the elliptic setting (Lemma 2.13 of \cite{KP2}). 

\subsection{The adjoint inhomogeneous Dirichlet problem}

 Let $u$ be the solution of the following boundary value problem
\begin{equation}
\mathcal{H}u=\partial_t u-\divg(A\nabla u)=0\mbox{ in }{\mathbb R}^{n+1}_+\times\mathbb R,\qquad u\Big|_{\partial({\mathbb R}^{n+1}_+\times\mathbb R)}=f,\label{e5xb}
\end{equation}
We find $\vec{h}$ as in Lemma \ref{l1bb} for the vector $-\nabla u$, i.e.,  $\|\wt N_{1,\varepsilon}(\nabla u)\|_{L^p}\lesssim -\iint_{{{\mathbb R}^{n+1}_+\times\mathbb R}}\nabla u\cdot\vec{h}\,dZ\,d\tau$.

We let $v:{\mathbb R}^{n+1}_+\times \R\to\mathbb R$ be the solution of the inhomogeneous Dirichlet problem for the operator ${\LL}^*$ (adjoint to $\LL$) which is a backward-in-time parabolic operator:
\begin{equation}
{ \LL}^*v=-\partial_tv-\divg(A^*\nabla v)=\divg(\vec{h})\mbox{ in }{\mathbb R}^{n+1}_+\times\mathbb R,\qquad v\Big|_{\partial({\mathbb R}^{n+1}_+\times\mathbb R)}=0.\label{e3}
\end{equation}
From results of \cite{Up,U} we have the square function and nontangential maximal function estimates for $v$
of the form $\|\tilde{N}(v)\|_{L^{p'}}\lesssim 1$, $\|S(v)\|_{L^{p'}}\lesssim 1$ for all $p'$ for which the Dirichlet problem for the adjoint operator ${ \LL}^*$ is solvable. (By \cite{AEN2}, this is true for all $p'>p_0$,
where $p_0>1$.)

\subsection{Proof of Theorem \ref{MainT} modulo certain bounds for square/nontangential maximal functions}

$\mbox{ }$\medskip

\noindent\textbf{Roadmap of this subsection.} The argument below proceeds in three steps. First, the duality lemma (Lemma~\ref{l1bb}) reduces the $L^p$ bound for $\widetilde{N}_{1,\varepsilon}(\nabla u)$ to a pairing $\iint \nabla u\cdot\vec h$ with a test vector field $\vec h$ supported in the interior. Second, by integrating by parts and using the inhomogeneous adjoint equation $\mathcal H^* v = \divg\,\vec h$, this pairing is rewritten as a boundary integral involving only the boundary data $f$, the test field $\vec h$, and the adjoint solution $v$. Third, a fundamental-theorem-of-calculus trick together with the insertion of the resolvent regularization $\mathcal P_\lambda f$ produces a collection of solid-integral pairings on $\R^{n+1}_+\times\R$. Each of these is then estimated by Cauchy--Schwarz against the square-function and nontangential-maximal-function bounds for $v$ (available because the adjoint Dirichlet problem is solvable in $L^{p'}$), leaving as residue a finite list of estimates for operators built out of $\mathcal P_\lambda$ acting on $f$. These outstanding estimates are collected in \eqref{list} below and proved in Sections~\ref{SA}, \ref{SN}, and~\ref{lemma:CarelsonBound}--\ref{lemma:AreaFctOfDiff}.\medskip

Fix $\varepsilon>0$. Our aim is to estimate $N_{1,\varepsilon}(\nabla u)$  in $L^p$ using the lemma above. Then since $\vec{h}\big|_{\partial({{\mathbb R}^{n+1}_+\times\mathbb R})}=0$ and $\vec{h}$ vanishes as $x_{n+1}\to\infty$, we have by integration by parts
\begin{equation}\label{e6}
\|\wt N_{1,\varepsilon}(\nabla u)\|_{L^p}\lesssim -\iint_{{{\mathbb R}^{n+1}_+\times\mathbb R}}\nabla u\cdot\vec{h}\,dZ\,d\tau
=\iint_{{\mathbb R}^{n+1}_+\times\mathbb R} u\,\divg\,\vec{h}\,dZ\,d\tau=
\iint_{{{\mathbb R}^{n+1}_+\times\mathbb R}}u\,{\LL}^*v\,dZ\,d\tau
\end{equation}
We now move $u$ inside the divergence operator and apply the divergence theorem.
\begin{equation}\nonumber
\mbox{RHS of \eqref{e6}}=-\iint_{{{\mathbb R}^{n+1}_+\times\mathbb R}}\divg(uA^*\nabla v)\,dZ\,d\tau+\iint_{{\mathbb R}^{n+1}_+\times\mathbb R} A\nabla u\cdot\nabla v\,dZ\,d\tau+\iint_{{\mathbb R}^{n+1}_+\times\mathbb R} (\partial_tu)v\,dZ\,d\tau
\end{equation}
$$=
\int_{\partial{{(\mathbb R}^{n+1}_+\times\mathbb R)}}u(\cdot,0,\cdot)a^*_{n+1,j}\partial_jv\,d\sigma,$$
since
$$\iint_{{\mathbb R}^{n+1}_+\times\mathbb R} A\nabla u\cdot\nabla v\,dZ\,d\tau+\iint_{{\mathbb R}^{n+1}_+\times\mathbb R} (\partial_tu)v\,dZ\,d\tau=-\iint_{{\mathbb R}^{n+1}_+\times\mathbb R} \mathcal Hu\,v\,dZ\,d\tau=0.$$
Here there is no boundary integral  involving $v$ since $v$ vanishes on the boundary of ${{\mathbb R}^{n+1}_+\times\mathbb R}$. It follows that
\begin{equation}\label{e8c}
\|N_{1,\varepsilon}(\nabla u)\|_{L^p}\lesssim
\int_{\partial{{(\mathbb R}^{n+1}_+\times\mathbb R)}}u(\cdot,0,\cdot)a^*_{n+1,j}\partial_jv\,d\sigma,
\end{equation}
where the implied constant in \eqref{e8c} is independent of $\varepsilon>0$.

Now, we use the fundamental theorem of calculus and the decay of $\nabla v$ at infinity to write \eqref{e8c} as
\begin{equation}\label{e9}
\|N_{1,\varepsilon}(\nabla u)\|_{L^p}\lesssim-\int_{\partial({{\mathbb R}^{n+1}_+}\times\mathbb R)}u(x,0,t)\left(\int_0^\infty \frac{d}{ds}\left(a^*_{n+1,j}(x,s,t)\partial_jv(x,s,t)\right)ds\right)dx\,dt.
\end{equation}
Recall that $\divg(A^*\nabla v)=-\divg(\vec{h})-\partial_tv$ and hence the righthand side of \eqref{e9} equals 
\begin{equation}\label{e10}
\int_{\partial{{(\mathbb R}^{n+1}_+}\times\mathbb R)}u(x,0,t)\Bigg(\int_0^\infty \Big[\sum_{i<n+1}\partial_i(a^*_{ij}(x,s,t)\partial_jv(x,s,t))+\divg\,\vec{h}(x,s,t)
\end{equation}
\begin{equation}\nonumber
+\partial_tv(x,s,t)\Big]ds\Bigg)dx\,dt.
\end{equation}
We integrate by parts moving $\partial_i$ for $i<n+1$ onto $u(\cdot,0,\cdot)$. 

The integral term containing $\partial_{n+1}h_{n+1}(x,s,t)$ does not need to be considered as it equals zero by the fundamental theorem of calculus since $\vec{h}(\cdot,0,\cdot)=\vec{0}$ and
$\vec{h}(\cdot,s,\cdot)\to\vec{0}$ as $s\to\infty$. 

The notation we will introduce below is: for a vector $\vec{w}=(w_1,w_2,\dots,w_n,w_{n+1})$, the vector $\vec{w}_\parallel$ denotes the first $n$ components of $\vec{w}$, that is $(w_1,w_2,\dots,w_{n})$, and $\nabla_{||}=\nabla_{x}$.
It follows that
\begin{eqnarray}\nonumber
&&\|N_{1,\varepsilon}(\nabla u)\|_{L^p}\\\nonumber&\lesssim& \int_{\partial({{\mathbb R}^{n+1}_+}\times\mathbb R)} \nabla_\parallel f(x,t)\cdot\left(\int_0^\infty \left[-\vec{h}_\parallel(x,s,t)-(A^*\nabla v)_\parallel(x,s,t)\right]ds\right)dx\,dt
\\\nonumber&+& \int_{\partial({{\mathbb R}^{n+1}_+}\times\mathbb R)} f(x,t)\left(\int_0^\infty \partial_tv(x,s,t)\,ds\right)dx\,dt
\\
&=&I+II+III.\label{e11}
\end{eqnarray}
Here $I$ is the term containing $\vec{h}_\parallel$ and $II$ contains $(A^*\nabla v)_\parallel$ and $III$ term $\partial_tv$.

For the term $I$ we observe 
\begin{multline}\label{termI}
I=-\iint_{\R^{n+1}_+\times\R} \nabla_{\parallel}f(x,t)\cdot \vec{h}_\parallel(x,s,t)\,ds\,dx\,dt
\lesssim \|\tilde{N}_1(\nabla_\parallel f)\|_{L^p(\partial({{\mathbb R}^{n+1}_+}\times\mathbb R))}\\
\le \|M(\nabla_\parallel f)\|_{L^p(\partial({{\mathbb R}^{n+1}_+}\times\mathbb R))}\lesssim \|\nabla_\parallel f\|_{L^p(\partial({{\mathbb R}^{n+1}_+}\times\mathbb R))}.
\end{multline}
Here we have used  \eqref{e2} and $L^p$ boundedness of the Hardy-Littlewood maximal function, denoted by $M$.

We now look at the term $II$. We introduce the notation $\lambda=x_{n+1}$ and hence
$\partial_\lambda=\partial_{n+1}$. With this notation, set
\begin{equation}
{V}(x,\lambda,t)=-\int_{\lambda}^\infty v(x,s,t)ds.
\end{equation}
We now introduce $\tilde{u}=\mathcal P_\lambda f$. Here $\mathcal P_\lambda$ is the operator
in variables $(x,t)$ with $\lambda$ fixed defined by $\mathcal P_\lambda=(I+\lambda^2\mathcal H_\parallel)^{-m}$ for sufficiently large $m>0$ to be determined later. Occasionally, we also work with its adjoint
$\mathcal P^*_\lambda=(I+\lambda^2\mathcal H^*_\parallel)^{-m}$. Here, and throughout the paper, the operators
$\mathcal H_\parallel$ and $\mathcal H^*_\parallel$ should be understood to denote operators acting in variables $(x,t)$
defined by
$$\LL_\parallel=\partial_t-\divg_\parallel(A_\parallel\nabla_\parallel\cdot), \qquad
\LL^*_\parallel=-\partial_t-\divg_\parallel(A^*_\parallel\nabla_\parallel\cdot),$$
where $A_\parallel$ is an $n\times n$ matrix valued function consisting of first $n\times n$ block
of the $n+1\times n+1$ matrix $A$. Similarly $\divg_\parallel$, $\nabla_\parallel$ are operators acting in $n$-variables, while the usual $\divg$, $\nabla$ act in $n+1$ variables (i.e., including the variable $x_{n+1}=\lambda$). In particular we note that \((A^*\nabla v)_{||}=A^*_{||}\nabla_{||}v + \sum_{i<n+1} a^*_{i,n+1}\partial_\lambda v\).

Since $\tilde{u}$ equals $f$ on the boundary
of our domain we have using the independence of $A$ on the variable $x_{n+1}=\lambda$:

\begin{multline}
    II=\int_{\partial{{(\mathbb R}^{n+1}_+\times\mathbb R)}}\nabla_\parallel f(x,0,t)\cdot (A^*\nabla {V})_\parallel(x,0,t)dx\,dt\\
 =  \int_{\partial{{(\mathbb R}^{n+1}_+\times\mathbb R)}}\nabla_\parallel f(x,0,t)\cdot A_\parallel^*\nabla_\parallel{V}(x,0,t)dx\,dt
 \\+ \sum_{i<n+1}\int_{\partial{{(\mathbb R}^{n+1}_+\times\mathbb R)}}\partial_i f(x,0,t) a^*_{i,n+1}(x,t)\int_0^\infty \partial_{\lambda}v(x,s,t)ds\,dx\,dt
 \\
    =-\iint_{\R^{n+1}_+\times\R}\dr_\lambda\br{\nabla_{||} \mathcal P_\lambda f(x,\lambda,t)\cdot A_\parallel^*\nabla_\parallel V(x,\lambda,t)}\dr_\lambda(\lambda)dxd\lambda dt\\
    =\iint_{{{\mathbb R}^{n+1}_+\times\mathbb R}}\partial^2_{\lambda}\left[\nabla_\parallel {\mathcal P_\lambda f}(x,\lambda,t)\cdot A_\parallel^*\nabla_\parallel\vec{V}(x,\lambda,t)\right]\lambda\,dx\,d\lambda\,dt.
\end{multline}
The term in the third line vanishes since $v(x,0,t)=0$ and $v(x,s,t)\to 0$ as $s\to\infty$.
Therefore,
\begin{multline*}
II=\iint_{{{\mathbb R}^{n+1}_+\times\mathbb R}}\partial^2_{\lambda}(\nabla_\parallel {\mathcal P_\lambda f})\cdot A^*_\parallel\nabla_\parallel{V}\lambda\,dx\,d\lambda\,dt\\+2\iint_{{{\mathbb R}^{n+1}_+\times\mathbb R}}\partial_{\lambda}(\nabla_\parallel {\mathcal P_\lambda f})\cdot A^*_\parallel\partial_\lambda(\nabla_\parallel{V})\lambda\,dx\,d\lambda\,dt+\\+\iint_{{{\mathbb R}^{n+1}_+}\times\mathbb R}\nabla_\parallel {\mathcal P_\lambda f}\cdot A^*_\parallel\partial^2_{\lambda}(\nabla_\parallel{V})\lambda\,dx\,d\lambda\,dt
=II_1+II_2+II_3.
\end{multline*}
We write term $III$ in a similar way as 
\begin{multline*}
    III=-\int_{\partial{{(\mathbb R}^{n+1}_+\times\mathbb R)}} f(x,t)\partial_t V(x,0,t)dx\,dt\\
    =-\iint_{{{\mathbb R}^{n+1}_+\times\mathbb R}}\partial^2_{\lambda}\left[\mathcal P_\lambda f(x,\lambda,t)\partial_t{V}(x,\lambda,t)\right]\lambda\,dx\,d\lambda\,dt\\
 =-\iint_{{{\mathbb R}^{n+1}_+\times\mathbb R}}(\partial^2_{\lambda} {\mathcal P_\lambda f})\partial_tV\lambda\,dx\,d\lambda\,dt-2\iint_{{{\mathbb R}^{n+1}_+\times\mathbb R}}(\partial_{\lambda}\mathcal P_\lambda f)\partial_\lambda(\partial_t{V})\lambda\,dx\,d\lambda\,dt\nonumber\\-\iint_{{{\mathbb R}^{n+1}_+}\times\mathbb R}{\mathcal P_\lambda f} \partial^2_{\lambda}(\partial_t{V})\lambda\,dx\,d\lambda\,dt
=III_1+III_2+III_3.   
\end{multline*}
Since $ \partial_\lambda{V}(x,\lambda,t)=v(x,\lambda,t)$, by exchanging the order of derivatives and integration by parts we get that
\begin{multline*}
II_3+III_3=\iint_{\R^{n+1}_+\times\R}\left[-\divg_\parallel (A_\parallel\nabla_\parallel \mathcal P_\lambda f)+\partial_t \mathcal P_\lambda f\right](\partial_\lambda v)\lambda\,dx\,d\lambda\,dt=\\
=\iint_{\R^{n+1}_+\times\R}(\mathcal H_\parallel \mathcal P_\lambda f)(\partial_\lambda v)\lambda\,dx\,d\lambda\,dt\le \Vert \mathcal{A}(\lambda\mathcal{H}_{||}\mathcal{P}_{\lambda}f)\Vert_{L^p}\Vert S(v)\Vert_{L^{p'}}\lesssim \Vert \mathcal{A}(\lambda\mathcal{H}_{||}\mathcal{P}_{\lambda}f)\Vert_{L^p},
\end{multline*}
since as we have noted earlier $\Vert S(v)\Vert_{L^{p'}}\lesssim 1$. We shall postpone the task of estimating
$ \|\mathcal{A}(\lambda\mathcal{H}_{||}\mathcal{P}_{\lambda}f)\Vert_{L^p}$ for later (section \ref{SA}) after we collect all required bounds we need to prove for operators associated with $\mathcal{P}_{\lambda}f$.

For the term $II_2$ we have
\begin{multline*}
II_2
\le C\int_{\partial(\R^{n+1}_+\times\R)}\iint_{\Gamma(x,t)}\abs{\partial_{\lambda}(\nabla_\parallel \mathcal P_\lambda f )}\abs{\nabla_\parallel{v}}\lambda^{-n-1}dy\,ds\,d\lambda\,dx\,dt\\
\le C\|\mathcal A(\lambda\nabla_\parallel\partial_\lambda \mathcal{P}_{\lambda}f)\|_{L^p}\|S(v)\|_{L^{p'}}\lesssim \|\mathcal A(\lambda\nabla_\parallel\partial_\lambda \mathcal{P}_{\lambda}f)\|_{L^p},
\end{multline*}
and we note that we shall have to return to the term $\|\mathcal A(\lambda\nabla_\parallel\partial_\lambda \mathcal{P}_{\lambda}f)\|_{L^p}$ later (again in section \ref{SA}).
For the term $III_2$ we use the PDE $v$ satisfies. This gives us

\begin{multline*}
    III_2=-2\iint_{\R^{n+1}_+\times\R} (\lambda\partial_{\lambda}\mathcal{P}_{\lambda}f) \partial_tv dx d\lambda dt
    \\
    =2\iint_{\R^{n+1}_+\times\R} (\lambda\partial_{\lambda}\mathcal{P}_{\lambda}f) \mathrm{div}(A^*\nabla v) dx d\lambda dt +2 \iint_{\R^{n+1}_+\times\R} (\lambda\partial_{\lambda}\mathcal{P}_{\lambda}f) \mathrm{div}(\vec h) dx d\lambda dt
    \\
    =-2\iint_{{\R^{n+1}_+\times\R}}\lambda\nabla_{||}\partial_{\lambda}\mathcal{P}_{\lambda}f\cdot (A^* \nabla v)_\parallel dx d\lambda dt
    \\
    -2 \sum_i\iint_{{\R^{n+1}_+\times\R}} \partial_{\lambda}\big(\lambda\partial_{\lambda}\mathcal{P}_{\lambda}f\big) (a^*_{n+1,i}\partial_i v) dx d\lambda dt
    \\-2\iint_{{\R^{n+1}_+\times\R}} \left[\lambda\nabla_{||}\partial_{\lambda}\mathcal{P}_{\lambda}f\cdot h_{||} + \partial_{\lambda}\big(\lambda\partial_{\lambda}\mathcal{P}_{\lambda}f\big) h_{n+1}\right] dx d\lambda dt.
\end{multline*}    
Here we have integrated by parts (in $\nabla_\parallel$ and $\partial_\lambda$) to move the divergence onto the other term. Note that no boundary terms arise. We can now estimate each term separately, either using square function bounds for terms that contain $v$ or the bound \eqref{e2} for terms that contain components of  $\vec{h}$. This gives us
\begin{multline*}    
    III_2
    \lesssim \Vert \mathcal{A}(\lambda\nabla_{||}\partial_{\lambda}\mathcal{P}_{\lambda}f)\Vert_{L^p}\Vert S(v)\Vert_{L^{p'}} + \Vert \mathcal{A}(\lambda\partial^2_{\lambda}\mathcal{P}_{\lambda}f)\Vert_{L^p}\Vert S(v)\Vert_{L^{p'}}
    \\
   + \Vert \mathcal{A}(\partial_{\lambda}\mathcal{P}_{\lambda}f)\Vert_{L^p}\Vert S(v)\Vert_{L^{p'}} + \Vert \tilde{N}(\lambda\nabla_{||}\partial_{\lambda}\mathcal{P}_{\lambda}f)\Vert_{L^p}
    + \Vert \tilde{N}(\lambda\partial^2_{\lambda}\mathcal{P}_{\lambda}f)\Vert_{L^p} + \Vert \tilde{N}(\partial_{\lambda}\mathcal{P}_{\lambda}f)\Vert_{L^p}\\
    \lesssim \Vert \mathcal{A}(\partial_{\lambda}\mathcal{P}_{\lambda}f)\Vert_{L^p}
   +  \Vert \mathcal{A}(\lambda\nabla_{||}\partial_{\lambda}\mathcal{P}_{\lambda}f)\Vert_{L^p}+ \Vert \mathcal{A}(\lambda\partial^2_{\lambda}\mathcal{P}_{\lambda}f)\Vert_{L^p}\\
   +\Vert \tilde{N}(\lambda\nabla_{||}\partial_{\lambda}\mathcal{P}_{\lambda}f)\Vert_{L^p}+\Vert \tilde{N}(\lambda\partial^2_{\lambda}\mathcal{P}_{\lambda}f)\Vert_{L^p} + \Vert \tilde{N}(\partial_{\lambda}\mathcal{P}_{\lambda}f)\Vert_{L^p}.
\end{multline*}

The new terms to deal with are two square functions $\Vert \mathcal{A}(\partial_{\lambda}\mathcal{P}_{\lambda}f)\Vert_{L^p}$ and $\Vert \mathcal{A}(\lambda\partial^2_{\lambda}\mathcal{P}_{\lambda}f)\Vert_{L^p}$
and three nontangential maximal functions $\Vert \tilde{N}(\partial_{\lambda}\mathcal{P}_{\lambda}f)\Vert_{L^p}$,
$\Vert \tilde{N}(\lambda\nabla_{||}\partial_{\lambda}\mathcal{P}_{\lambda}f)\Vert_{L^p}$, $\Vert \tilde{N}(\lambda\partial^2_{\lambda}\mathcal{P}_{\lambda}f)\Vert_{L^p}$. These area function bounds are estimated in section \ref{SA} and the nontangential bounds in section \ref{SN}.
\medskip

Lastly, we consider $II_1+III_1$ together, moving first $\nabla_\parallel$ in the term $II_1$:
 \begin{multline*} 
    II_1+III_1=-\iint_{\R^{n+1}_+\times\R} (\partial^2_{\lambda}\mathcal{P}_{\lambda}f)\left[\divg_\parallel (A^*_\parallel\nabla_\parallel V)+\partial_tV\right]\lambda dxd\lambda dt\\
 =  \iint_{\R^{n+1}_+\times\R} (\partial^2_{\lambda}\mathcal{P}_{\lambda}f)(\mathcal H_\parallel^*V)\lambda dxd\lambda dt.
 \end{multline*}   
Since we prefer to have the full operator $\mathcal H$ on $V$ we add the missing components in the expression. This gives us the terms    
 \begin{multline*} 
   II_1+III_1  =\iint_{\R^{n+1}_+\times\R} (\partial^2_{\lambda}\mathcal{P}_{\lambda}f)(\mathcal{H}^*V)\lambda dx d\lambda dt - \sum_{j=1}^n\iint_{\R^{n+1}_+\times\R}(\partial^2_{\lambda}\mathcal{P}_{\lambda}f)\partial_j(a^*_{j,n+1}\partial_{\lambda}V)\lambda\, dx d\lambda dt
    \\
     - \sum_{j=1}^{n+1}\iint_{\R^{n+1}_+\times\R}(\partial^2_{\lambda}\mathcal{P}_{\lambda}f)a^*_{n+1,j}(\partial_j\partial_{\lambda}V)\lambda dx d\lambda dt
    :=IV_1+IV_2+IV_3.
\end{multline*}
The last term ($IV_3$) is not problematic at it enjoys a bound by 
\begin{align*}
    IV_3\lesssim \Vert \mathcal{A}(\lambda\partial^2_{\lambda}\mathcal{P}_{\lambda}f)\Vert_{L^p}\Vert S(v)\Vert_{L^{p'}}\lesssim  \Vert \mathcal{A}(\lambda\partial^2_{\lambda}\mathcal{P}_{\lambda}f)\Vert_{L^p},
\end{align*}
which is a term we have already encountered.
For the first term we use the idea of \cite{HKMP2} (p.888). We write the definition of $V$
which involves the integral $-\int_{\lambda}^\infty\dots ds$. We also integrate in the term $IV_1$ in $\lambda$, i.e., 
we have $-\int_0^\infty \int_{\lambda}^\infty\dots ds\,d\lambda$ and now we use Fubini and exchange the order of these two integrals to obtain $-\int_0^\infty \int_{0}^s\dots ,d\lambda\,ds$. This leads to (after swapping the roles of $\lambda$ and $s$):
\begin{align*}
    IV_1&=\iint_{\R^{n+1}_+\times\R} \left[\left(\int_0^{\lambda}s\partial^2_{s}\mathcal{P}_sf ds\right)\divg(A^*\nabla v) + \left(\int_0^{\lambda}s\partial^2_{s}\mathcal{P}_sf ds\right)\partial_tv\right] \,dx \,d\lambda \,dt.
\end{align*}  
Recall that $v$ solves the inhomogeneous PDE \eqref{e3}. Using it and then moving divergence onto the other term we obtain
\begin{multline*}
    IV_1=\iint_{\R^{n+1}_+\times\R} h\cdot\nabla\left(\int_0^{\lambda}s\partial^2_{s}\mathcal{P}_sf ds\right) \,dx \,d\lambda \,dt
    \\
    =\iint_{\R^{n+1}_+\times\R} \left[h_{||}\cdot\nabla_{||}\big(\lambda\partial_{\lambda}\mathcal{P}_{\lambda}f-\mathcal{P}_{\lambda}f+f\big)  + h_{n+1} \lambda\partial_{\lambda}^2\mathcal{P}_{\lambda}f\right]\,dx \,d\lambda \,dt.
\end{multline*}    
 We may now again use \eqref{e2} to bound each term. It follows that
 \begin{multline*}
 IV_1   \lesssim \Vert \tilde{N}(\lambda\nabla_{||}\partial_{\lambda}\mathcal{P}_{\lambda}f)\Vert_{L^p}+\Vert \tilde{N}(\nabla_{||}\mathcal{P}_{\lambda}f)\Vert_{L^p}+\Vert\tilde{N}(\lambda\partial^2_{\lambda}\mathcal{P}_{\lambda}f)\Vert_{L^p} + \|\tilde N(\nabla_\parallel f)\|_{L^p}.
\end{multline*}
The only new term is $\Vert \tilde{N}(\nabla_{||}\mathcal{P}_{\lambda}f)\Vert_{L^p}$, which we add to our list of required bounds to be estimated later in section \ref{SN}. 

It remains to consider  term $IV_2$. Let  $\mathfrak{A}_\lambda$ be the following averaging operator.
\begin{equation}\label{def-frak}
\mathfrak{A}_\lambda f(x,t)=\displaystyle\fint_{Q_\lambda(x,t)} f(y,s)\,dyds.
\end{equation}
Here $Q_\lambda(x,t)$ is a boundary parabolic cube centered at $(x,t)$ of size $\lambda$.
Occasionally, we apply the operator $\mathfrak{A}_\lambda$ to functions that depend on $\lambda$ as well,
and by $\mathfrak{A}_\lambda f(x,\lambda,t)$ we understand $\mathfrak{A}_\lambda f_\lambda(x,t)$,
where $f_\lambda=f(\cdot,\lambda,\cdot)$. This averaging operator will even out oscillation of the function $v$
to which it will be applied. Essentially, thanks to it we will be able to treat $\mathfrak{A}_\lambda v$
as being essentially constant on sets of size $\lambda$, which will be a key step in proving the commutator estimate listed as 4) in \eqref{list}. Simultaneously, we shall prove good bounds on $\Vert \mathcal{A}((I-\mathfrak{A}_{\lambda})v)\Vert_{L^{p'}}$ and a trivial bound $\Vert N(\mathfrak{A}_{\lambda}v)\Vert_{L^{p'}}\lesssim \Vert N(v)\Vert_{L^{p'}}$ which allows us to insert this operator into the calculation below.
For the  term $IV_2$ we have the following:
\begin{multline*}
    IV_2=-\sum_{j=1}^n\iint_{\R^{n+1}_+\times\R}(\partial^2_\lambda\mathcal P_\lambda f)\partial_j(a_{n+1,j}v)\lambda dx d\lambda dt\\
    =-\sum_{j=1}^n\iint_{\R^{n+1}_+\times\R}(\partial^2_\lambda\mathcal P_\lambda f)\partial_j(a_{n+1,j}(I-\mathfrak{A}_\lambda)v)\lambda dx d\lambda dt
    \\-\sum_{j=1}^n\iint_{\R^{n+1}_+\times\R}(\partial^2_\lambda\mathcal P_\lambda f)\partial_j(a_{n+1,j}\mathfrak{A}_\lambda v)\lambda dx d\lambda dt
    \\
    = \sum_{j=1}^n\iint_{\R^{n+1}_+\times\R}
    (\partial_j\partial^2_{\lambda}\mathcal{P}_{\lambda}f)(a_{n+1,j }(I-\mathfrak{A}_{\lambda})v)\lambda dx d\lambda dt \\- \iint_{\R^{n+1}_+\times\R}(\partial^2_{\lambda}\mathcal{P}_{\lambda}f)\partial_j(a_{n+1,j}\mathfrak{A}_{\lambda}v)\lambda dx d\lambda dt
       =:V_1+V_2.
\end{multline*}

By H\"older's inequality  we obtain
\begin{align*}
    V_1\lesssim \Vert \mathcal{A}(\lambda^2\nabla_{||}\partial^2_{\lambda}\mathcal{P}_{\lambda}f)\Vert_{L^p}\Vert \mathcal{A}((I-\mathfrak{A}_{\lambda})v)\Vert_{L^{p'}}\lesssim \Vert \mathcal{A}(\lambda^2\nabla_{||}\partial^2_{\lambda}\mathcal{P}_{\lambda}f)\Vert_{L^p},
\end{align*}
since the second term is bounded by $O(1)$ thanks to Corollary \ref{cor-mathfrak} and the first term will be considered in section \ref{SA}.
\medskip

For $V_2$ we first observe the following. The operator $\partial^2_{\lambda}\mathcal{P}_{\lambda}$
can be written as $\tilde{\mathcal{P}}_\lambda\circ \mathcal{Q}_\lambda$, where $\tilde{\mathcal{P}}_\lambda$
is of the same type as the original operator ${\mathcal{P}}_\lambda$ (but with a different power of the resolvent, which does not change its properties). The second operator 
\(\mathcal{Q}_\lambda\) is a linear combination of operators $\partial_\lambda\mathcal{P}_{\lambda}$ for different $m$ multiplied by the factor \(\frac{1}{\lambda}\) as follows from the first line of \eqref{P-lambda22}. Therefore any square function bound we have for  \(\partial_\lambda\mathcal{P}_{\lambda}\)  will also hold for 
 \(\lambda^{-1}\mathcal{Q}_\lambda\)  as long as we choose a  sufficiently large $m$ initially.
To see this we now explicitly track the dependence of $\mathcal P_\lambda$
on $m$, that is we write $\mathcal P_\lambda=\mathcal P_{\lambda,m}=(I+\lambda^2\mathcal H_\parallel)^{-m}$. Then we have:
\begin{multline}\label{P-lambda22} \partial^2_\lambda \mathcal{P}_{\lambda,m}=
\frac{2m}\lambda\left[\partial_\lambda\mathcal P_{\lambda,m+1}-\partial_\lambda\mathcal P_{\lambda,m}\right]=\\
\frac{4m}{\lambda^2}\left[(m+1)(\mathcal P_{\lambda,m+2}-\mathcal P_{\lambda,m+1})-m(\mathcal P_{\lambda,m+1}-\mathcal P_{\lambda,m})\right]\\
=(I+\lambda^2\mathcal H_\parallel)^{-k}\circ\frac{4m}{\lambda^2}\left[(m+1)(\mathcal P_{\lambda,m-k+2}-\mathcal P_{\lambda,m-k+1})-m(\mathcal P_{\lambda,m-k+1}-\mathcal P_{\lambda,m-k})\right]\\
=\mathcal{P}_{\lambda,{k}}\circ\frac{2m}{\lambda}\left[\frac{m+1}{m+1-k}\partial_\lambda\mathcal P_{\lambda,m-k-1}-\frac{m}{m-k}\partial_\lambda\mathcal P_{\lambda,m-k-2}\right]
:=\tilde{\mathcal{P}}_\lambda\circ \mathcal{Q}_\lambda.
\end{multline}
So indeed, the claim holds, provided we choose an integer $k\sim m/2$. We now move  $\tilde{\mathcal{P}}_\lambda$ onto $\partial_j(a_{n+1,j}\mathfrak{A}_{\lambda}v)$ in $V_2$ to gain more regularity for this term
since at the moment it only makes sense as a distribution.
It follows that for $V_2$ we have:
\begin{multline*}
    V_2=-\sum_{j=1}^{n}\iint_{\R^{n+1}_+\times\R}(\tilde{\mathcal{P}}_\lambda\circ \mathcal{Q}_\lambda f)\partial_j(a_{n+1,j}\mathfrak{A}_{\lambda}v)\lambda dx d\lambda dt
    \\
    = -\sum_{j=1}^n\iint_{\R^{n+1}_+\times\R}(\mathcal{Q}_{\lambda}f)(\tilde{\mathcal{P}}^*_{\lambda}\partial_j(a_{n+1,j}\mathfrak{A}_{\lambda}v))\lambda dx d\lambda dt
    \\
    =-\sum_{j=1}^n\iint_{\R^{n+1}_+\times\R}(\mathcal{Q}_{\lambda}f)\Big(\tilde{\mathcal{P}}^*_{\lambda}\partial_j(a_{n+1,j}\mathfrak{A}_{\lambda}v) - (\tilde{\mathcal{P}}^*_{\lambda}\partial_ja_{n+1,j})(\mathfrak{A}_{\lambda}v)\Big)\lambda dx d\lambda dt
    \\
     - \iint_{\R^{n+1}_+\times\R}(\mathcal{Q}_{\lambda}f)(\tilde{\mathcal{P}}^*_{\lambda}\partial_ja_{n+1,j})(\mathfrak{A}_\lambda v) \lambda \,dx d\lambda dt =:VI_1+VI_2.
\end{multline*}

For the first term we have by Cauchy-Schwarz:
\begin{align*}
    VI_1&\lesssim  \Vert\mathcal{A}(\lambda\mathcal{Q}_{\lambda}f)\Vert_{L^p}\Vert \mathcal{A}\Big(\lambda\big(\tilde{\mathcal{P}}^*_{\lambda}\partial_j(a_{n+1,j}\mathfrak{A}_{\lambda}) - (\tilde{\mathcal{P}}^*_{\lambda}\partial_j a_{n+1,j})\mathfrak{A}_{\lambda}\big)v\Big)\Vert_{L^{p'}}.
\end{align*}
As we have said earlier, whatever estimate we establish for $\Vert\mathcal{A}(\partial_\lambda\mathcal P_\lambda f)\Vert_{L^p}$ will also hold for $\Vert\mathcal{A}(\lambda\mathcal{Q}_{\lambda}f)\Vert_{L^p}$ and hence we don't have to add this term to our list as the term $\Vert\mathcal{A}(\partial_\lambda\mathcal P_\lambda f)\Vert_{L^p}$ is already on it. The second term is new and we will seek an estimate for it of by $O(1)$.

Lastly, we have
\begin{align*}
    VI_2&\lesssim \sum_{j=1}^n\Vert \mathcal{A}(\lambda\mathcal{Q}_{\lambda}f)\Vert_{L^p}\Vert C(\lambda\tilde{\mathcal{P}}^*_{\lambda}\partial_ja_{n+1,j})\Vert_{L^\infty}\Vert N(\mathfrak{A}_{\lambda}v)\Vert_{L^{p'}}.
\end{align*}
We can observe that by definition of the averaging operator \(\Vert N(\mathfrak{A}_{\lambda}v)\Vert_{L^{p'}}\lesssim \Vert {N}(v)\Vert_{L^{p'}}\) which is bounded by $O(1)$. Again the term $\Vert \mathcal{A}(\lambda\mathcal{Q}_{\lambda}f)\Vert_{L^p}$ has appeared before. Hence, the only outstanding term is
$\Vert C(\lambda\tilde{\mathcal{P}}^*_{\lambda}\partial_ja_{n+1,j})\Vert_{L^\infty}$.

\subsection{Summary: the list of outstanding estimates}
In summary, let us assume that for some $p>1$ we can prove the following bounds for an arbitrary function $f\in\dot{L}^p_{1,1/2}(\R^{n+1}_+\times\R)$ and the adjoint solution $v$ to \eqref{e3}:
\begin{align}
1)&\nonumber \qquad \Vert \mathcal{A}(\partial_{\lambda}\mathcal{P}_{\lambda}f)\Vert_{L^p}+\|\mathcal{A}(\lambda\mathcal{H}_{||}\mathcal{P}_{\lambda}f)\Vert_{L^p}+\|\mathcal A(\lambda\nabla_\parallel\partial_\lambda \mathcal{P}_{\lambda}f)\|_{L^p}\\
&\qquad\qquad\qquad\qquad+\Vert \mathcal{A}(\lambda\partial^2_{\lambda}\mathcal{P}_{\lambda}f)\Vert_{L^p}+\Vert \mathcal{A}(\lambda^2\nabla_{||}\partial^2_{\lambda}\mathcal{P}_{\lambda}f)\Vert_{L^p}\lesssim \|f\|_{\dot{L}^p_{1,1/2}},\nonumber\\
&\,\quad\qquad\qquad\qquad\qquad\qquad\qquad\qquad\qquad\mbox{to be established in section \ref{SA},}\nonumber\\
2)&\qquad\nonumber \Vert {N}(\partial_{\lambda}\mathcal{P}_{\lambda}f)\Vert_{L^p}+
\Vert \tilde{N}(\lambda\nabla_{||}\partial_{\lambda}\mathcal{P}_{\lambda}f)\Vert_{L^p}\nonumber\\
&\qquad\qquad\qquad\qquad\qquad\qquad+\Vert {N}(\lambda\partial^2_{\lambda}\mathcal{P}_{\lambda}f)\Vert_{L^p}+\Vert \tilde{N}(\nabla_{||}\mathcal{P}_{\lambda}f)\Vert_{L^p}\lesssim \|f\|_{\dot{L}^p_{1,1/2}}\label{list},\\
&\,\quad\qquad\qquad\qquad\qquad\qquad\qquad\qquad\qquad\mbox{to be established in section \ref{SN},}\nonumber\\
3)&\qquad\nonumber \Vert C(\lambda\tilde{\mathcal{P}}^*_{\lambda}\partial_ja_{n+1,j})\Vert_{L^\infty}\lesssim 1,\qquad\qquad\mbox{proven in Lemma \ref{lemma:CarelsonBound},}\\
4)&\qquad\nonumber \Vert \mathcal{A}\Big(\lambda\big(\tilde{\mathcal{P}}^*_{\lambda}\partial_j(a_{n+1,j}\mathfrak{A}_{\lambda}) - (\tilde{\mathcal{P}}^*_{\lambda}\partial_j a_{n+1,j})\mathfrak{A}_{\lambda}\big)v\Big)\Vert_{L^{p'}}\lesssim 1,\\
&\,\quad\qquad\qquad\qquad\qquad\qquad\qquad\qquad\qquad\mbox{proven as Corollary \ref{c712}.}\nonumber
\end{align}

Then for any $\varepsilon>0$ the estimate 
$$\|\wt N_{1,\varepsilon}(\nabla u)\|_{L^p}\le C\|f\|_{\dot{L}^p_{1,1/2}},$$
holds for the solution of the Regularity problem with boundary datum $f$.

An argument is required to demonstrate that the control of $\tilde{N}_{1,\varepsilon}(\nabla u)$ of a solution $\LL u=0$ implies the control of $\tilde{N}(\nabla u)$ (the $L^2$ averaged version of the non-tangential maximal function as defined in \eqref{def.N2}). Firstly, as the established estimates are independent of $\varepsilon>0$ we obtain
$$\|\tilde{N}_{1}(\nabla u)\|_{L^p}=\lim_{\varepsilon\to 0+}\|\tilde{N}_{1,\varepsilon}(\nabla u)\|_{L^p}\le C(\|\nabla_\parallel f\|_{L^p}+\|D^{1/2}_t f\|_{L^p}).$$
Secondly, 
by the higher integrability of the gradient (i.e. Gehring's lemma) and purely real variable arguments (such as \cite{She} of Shen), we have that
$$\left(\fiint_B|\nabla u|^2\right)^{1/2}\lesssim \left(\fiint_{2B}|\nabla u|\right),$$
which implies a bound of $\tilde{N}(\nabla u)(\cdot)$ defined using cones $\Gamma_a(\cdot)$ of some aperture $a>0$ by $\tilde{N}_1(\nabla u)(\cdot)$ defined using cones $\Gamma_b(\cdot)$ of some slightly larger aperture $b>a$.
\medskip

Hence, if we establish \eqref{list} the claim of Theorem \ref{MainT} follows on the domain $\Omega=\R^{n+1}_+\times \R$.

The claim for the domain of the form $\Omega=\mathcal O\times\R$, where $\mathcal O\subset\R^{n+1}$ is an unbounded Lipschitz domain
$$\mathcal O=\{(x,x_{n+1}): x\in\R^n\mbox{ and }x_{n+1}>\phi(x)\},\quad\mbox{for some Lipschitz function $\phi$},$$
follows too. We just need to consider a change of variables 
$(x,\lambda,t)\mapsto( x,\lambda-\phi(x),t)$
which transforms the domain from $\mathcal O\times\R$ to the previously considered case $\R^{n+1}_+\times \R$. This transformation preserves the property of $\lambda=x_{n+1}$-independence and ellipticity and hence we have reduced the problem to the case we have already resolved.

\section{Area functions estimates}\label{SA}

The main goal of the section is to establish the bounds stated in 1) of \eqref{list}. Again, the underlying space is $\Omega=\R^{n+1}_+\times\R$ and hence $\partial\Omega=\R^n\times\R$.
The notation we use throughout this section is as follows:
\begin{enumerate}
\item  $\mathcal{P}_{\lambda}$ denotes $(I+\lambda^2\mathcal H_\parallel)^{-m}$ for any  sufficiently large $m>0$. Recall that $\mathcal H_\parallel=\partial_t-\divg_\parallel (A_\parallel\nabla_\parallel\cdot)$.
\item When the dependence of $\mathcal{P}_{\lambda}$ on $m$ matters, we shall use the previously introduced notation  $\mathcal P_{\lambda,m}$.
\end{enumerate}

\subsection{The area function $\mathcal A(\partial_\lambda\mathcal P_\lambda f)$}\label{s5.1}

We recall the result in \cite{AEN2} (Lemma 3.2 and its proof) which immediately gives that
\begin{equation}\label{area-l2}
\Vert \mathcal{A}(\partial_{\lambda}\mathcal{P}_{\lambda}f)\Vert_{L^2}\lesssim \|f\|_{\dot{L}^2_{1,1/2}}=\|\nabla_\parallel f\|_{L^2}+\|D^{1/2}_t f\|_{L^2}.
\end{equation}

We aim to extend the range for which this holds to $1<p\le 2$ by interpolation. Let us recall results of the paper \cite{DinN} which introduces the following Hardy-Sobolev space. In the definition below the variables are $(x,t)$ with $x\in\R^n$ and $t\in\R$. This Hardy-Sobolev space serves as the natural endpoint at $p=1$ for the Regularity problem, in place of the space $\dot L^1_{1,1/2}(\R^n\times\R)$, which fails to be a suitable endpoint precisely because its atoms do not satisfy the cancellation condition needed for the kernel-based off-diagonal estimates. As in the classical Fefferman--Stein theory, the Hardy-space atoms supply the orthogonality (vanishing moment) properties that allow Calder\'on--Zygmund-type operators to be controlled in $L^1$, whereas elements of $\dot L^1_{1,1/2}$ in general do not. Real interpolation between this $p=1$ Hardy--Sobolev endpoint and the $L^q$ scale (Proposition~\ref{IHS2} below) then recovers the full range of $\dot L^p_{1,1/2}$ spaces, which is exactly what is needed for our argument.

\begin{definition}\label{HSato} Let $1<\beta<\infty$. We say that a function $b:\R^n\times\R\to\R$ is a homogeneous $(1,1/2,\beta)$-atom associated to a parabolic ball $\Delta\subset\R^n\times\R$ of radius $r=r(\Delta)$ if 
\begin{itemize}
\item[(i)] $b$ is supported in the ball $\Delta$,
\item[(ii)] $\|\nabla b\|_{L^\beta(\R^n\times\R)}+\|D^{1/2}_t b\|_{L^\beta(\R^n\times\R)}\le |\Delta|^{-1/\beta'}$,
\item[(iii)] $\|b\|_{L^1(\R^n\times\R)}\le r(\Delta)$.
\end{itemize}
Here $|\Delta|$ denotes the usual $n+1$-dimensional Lebesgue measure on $\R^n\times\R$.
When $p=\infty$ we modify (ii) and replace it by $\|\nabla b\|_{L^\infty(\R^n\times\R)}+\|D^{1/2}_t b\|_{\BMO(\R^n\times\R)}\le |B|^{-1}$, where $\BMO$ is the parabolic BMO space.
\vglue2mm

We say that a locally integrable function $f$ belongs to the homogeneous parabolic atomic Hardy-Sobolev space
$\dot{\HS}^{1,(\beta)}_{1,1/2}(\R^n\times\R)$ if there exists a family of locally integrable homogeneous $(1,1/2,\beta)$-atoms $(b_i)_i$ such that $f=\sum_i\lambda_ib_i$, with $\sum_i|\lambda_i|<\infty$. We equip the space 
$\dot{\HS}^{1,(\beta)}_{1,1/2}(\R^n\times\R)$ with the norm:
$$\|f\|_{\dot{\HS}^{1,(\beta)}_{1,1/2}}=\inf_{(\lambda_i)_i}\sum_i|\lambda_i|.$$
\end{definition}

As shown in \cite{DinN}, for all $1<\beta\le\infty$, the spaces 
$\dot{\HS}^{1,(\beta)}_{1,1/2}(\R^n\times\R)$ are the same and the norms are equivalent. Hence we can drop the dependence on $\beta$. Furthermore, the space $\dot{\HS}^{1}_{1,1/2}(\R^n\times\R)$ interpolates well 
with the spaces  ${\dot{L}^p_{1,1/2}(\R^n\times\R)}$ since we have the following (by Proposition 6.56 of \cite{DinN}):

\begin{proposition} \label{IHS2} For all $1<p<q<\infty$ and $\theta\in (0,1)$ satisfying $\frac1p=(1-\theta)+\frac\theta{q}$,  $\dot{L}^{p}_{1,1/2}(\R^n\times\R)$ is a real interpolation space between $\dot{\HS}^{1}_{1,1/2}(\R^n\times\R)$ and $\dot{L}^{q}_{1,1/2}(\R^n\times\R)$. More precisely, we have
$$ \left(\dot{\HS}^{1}_{1,1/2}(\R^n\times\R),\dot{L}^{q}_{1,1/2}(\R^n\times\R)\right)_{\theta,p} = \dot{L}^{p}_{1,1/2}(\R^n\times\R),$$
with equivalent norms. 
\end{proposition}

It follows that if we prove that for any $(1,1/2,2)$-atom $b$ we have the following estimate
\begin{equation}\label{eq-atom}
\|\mathcal{A}(\partial_{\lambda}\mathcal{P}_{\lambda}b)\Vert_{L^1}\le C,
\end{equation}
for $C>0$ independent of $b$, then by real interpolation (applied to the sublinear functional 
$f\mapsto \mathcal{A}(\partial_{\lambda}\mathcal{P}_{\lambda}f))$ we can conclude that for all $1<p\le 2$ we have the $L^p$ bounds:
\begin{equation}\label{area-lp}
\Vert \mathcal{A}(\partial_{\lambda}\mathcal{P}_{\lambda}f)\Vert_{L^p}\lesssim \|f\|_{\dot{L}^p_{1,1/2}}=\|\nabla_\parallel f\|_{L^p}+\|D^{1/2}_t f\|_{L^p}.
\end{equation}

To establish \eqref{eq-atom}, consider any $(1,1/2,2)$-atom $b$. Assume it is supported on a boundary ball $\Delta$. We then have, using \eqref{area-l2}, Cauchy--Schwarz, and doubling,

\begin{multline}
\|\mathcal{A}(\partial_{\lambda}\mathcal{P}_{\lambda}b)\Vert_{L^1(8\Delta)}\le \left(\int_{8\Delta}1^2\,dxdt\right)^{1/2}\left(\int_{8\Delta}|\mathcal{A}(\partial_{\lambda}\mathcal{P}_{\lambda}b)|^2\,dxdt\right)^{1/2}
\\\le C|\Delta|^{1/2}\|\mathcal{A}(\partial_{\lambda}\mathcal{P}_{\lambda}b)\Vert_{L^2(\R^n\times \R)}\le C|\Delta|^{1/2}
\left(\|\nabla b\|_{L^2(\R^n\times\R)}+\|D^{1/2}_t b\|_{L^2(\R^n\times\R)}\right)\\\le C|\Delta|^{1/2}|\Delta|^{-1/2}\le C.
\end{multline}
In the penultimate step we have used the property (ii) of the $b$-atom. Hence, all that remains is to prove $L^1$ bounds away from the support of the atom, i.e., $\|\mathcal{A}(\partial_{\lambda}\mathcal{P}_{\lambda}b)\Vert_{L^1((\R^n\times\R)\setminus 8\Delta)}$. For this we use the following kernel bounds from \cite{AEN2} (Lemma 4.4).
\begin{lemma}\label{lemma:KernelBounds}
    The integral kernel of \(\mathcal P_{\lambda,m}=(I+\lambda^2\mathcal{H}_{||})^{-m}\) satisfies
    \begin{align}K_{\lambda,m}(x,t,y,s)\lesssim \frac{C\chi_{(0,\infty)}(t-s)}{\lambda^{2m}}(t-s)^{-\frac{n}{2}+m-1}e^{-\frac{t-s}{\lambda^2}}e^{-c\frac{|x-y|^2}{t-s}}\label{eq:kernelBounds}.
    \end{align}
\end{lemma}

As a corollary we obtain the following. Let $\tilde{K}_{\lambda,m}(x,t,y,s)$ be the integral kernel of the operator
\(\partial_\lambda\mathcal P_{\lambda,m}\). Then by \eqref{P-lambda} it enjoys the bounds:
    \begin{align}\left|\tilde{K}_{\lambda,m}(x,t,y,s)\right|\le \frac{C}{\lambda}\left[K_{\lambda,m+1}(x,t,y,s)+K_{\lambda,m}(x,t,y,s)\right].
    \end{align}
Hence any bounds established for an operator via bounds for the kernel $\lambda^{-1}K_{\lambda,m}(x,t,y,s)$
and for $m>0$ large will also apply to \(\partial_\lambda\mathcal P_{\lambda,m}\), which is simply a sum of two such operators. Without loss of generality, in what follows below will be working with $\lambda^{-1}K_{\lambda,m}(x,t,y,s)$.\vglue1mm

Let us decompose the set $(\R^n\times\R)\setminus 8\Delta$ and write it as a union of dyadic annuli, where $r=r(\Delta)$ denotes the radius of $\Delta$. All balls below are centered at the center of $\Delta$.

$$(\R^n\times\R)\setminus 8\Delta=\bigcup_{j\ge 4}\left(\Delta_{2^jr}\setminus \Delta_{2^{j-1}r}\right).
$$
Fix a point $(z,\tau)\in \Delta_{2^jr}\setminus \Delta_{2^{j-1}r}$ and consider the corresponding parabolic cone
$\Gamma_a(z,\tau)$. We split the cone further and consider the `near' and `away' part of the cone
\begin{align*}
\Gamma^{n}_a(z,\tau)&=\{(X,t)\in \Gamma_a(z,\tau): x_{n+1}<\beta2^j r\},\\
\Gamma^{a}_a(z,\tau)&=\{(X,t)\in \Gamma_a(z,\tau): x_{n+1}\ge \beta2^j r\}.
\end{align*}
Here $\beta=\beta(a)>0$ depends on the aperture $a$ of the cone and is chosen so that for any
$(X,t)=(x,\lambda,t)\in \Gamma^{n}_a(z,\tau)$, the parabolic distance between $(x,t)$ and $\Delta$ is at least  $2^{j-2}r$. It follows that for any $(X,t)\in \Gamma^{n}_a(z,\tau)$ we have
\begin{multline}\label{eq:pointwisepartialX0P}
(\lambda^{-1}\mathcal P_{\lambda,m}b)(x,\lambda,t)=
\int_{\R^n\times\R}\lambda^{-1}{K}_{\lambda}(x,t,y,s) b(y,s) dyds
        \\
        \leq\int_{\partial\Omega} \frac{C\chi_{(0,\infty)}(t-s)}{\lambda^{2m+1}}(t-s)^{-\frac{n}{2}+m-1}e^{-\frac{t-s}{\lambda^2}}e^{-c\frac{|x-y|^2}{t-s}} b(y,s) dyds
        \\
        =\int_{\Delta} \chi_{|t-s|\geq 2^{2j-5}r^2}\frac{C\chi_{(0,\infty)}(t-s)}{\lambda^{2m+1}}(t-s)^{-\frac{n}{2}+m-1}e^{-\frac{t-s}{\lambda^2}}e^{-c\frac{|x-y|^2}{t-s}} b(y,s) dyds\nonumber
        \\
        +\int_{\Delta} \chi_{|t-s|\leq 2^{2j-5}r^2}\frac{C\chi_{(0,\infty)}(t-s)}{\lambda^{2m+1}}(t-s)^{-\frac{n}{2}+m-1}e^{-\frac{t-s}{\lambda^2}}e^{-c\frac{|x-y|^2}{t-s}} b(y,s) dyds\nonumber
        \\
        :=J_1+J_2,\nonumber
\end{multline}

    where we are distinguishing two cases to simplify the computation.
    
    First for \(J_1\), since \(|t-s|\geq 2^{2j-5}r^2\), we can bound \(e^{-\frac{|x-y|^2}{t-s}}\lesssim 1\), and we have
    \[ \frac{C\chi_{(0,\infty)}(t-s)}{\lambda^{2m+1}}(t-s)^{-\frac{n}{2}+m-1}e^{-\frac{t-s}{\lambda^2}}e^{-c\frac{|x-y|^2}{t-s}}\lesssim \frac{C}{\lambda^{2m+1}}(2^{2j}r^2)^{-\frac{n}{2}+m-1}e^{-c\frac{2^{2j}r^2}{\lambda^2}}. \]
    
 Observe that the function \(\lambda\mapsto \lambda^{-4m-3}e^{-c\frac{2^{2j}r^2}{\lambda^2}}\) is maximized for \(\lambda=\sqrt{\frac{2c}{4m+3}}2^{j}r\). This yields for $(z,\tau)\in \Delta_{2^jr}\setminus \Delta_{2^{j-1}r}$
    \begin{align*}
        \iint_{\Gamma^n(z,\tau)}\frac{|J_1|^2}{\lambda^{n+3}}dx d\lambda dt&\lesssim \int_0^{\beta 2^j r}\frac{C}{\lambda^{4m+3}}(2^jr)^{-2n+4m-4}e^{-c\frac{2^{2j}r^2}{\lambda^2}}\Vert b\Vert_{L^1}^2d\lambda
        \\
        &\lesssim \Vert b\Vert_{L^1}^2\int_0^{\beta2^jr} (2^jr)^{-2n-7}d\lambda\lesssim 2^{-2(n+3)j}r^{-2(n+3)}\Vert b\Vert_{L^1}^2.
    \end{align*}
    We have estimated the function inside the integral by its supremum.  Here the implicit constant depends on $m$ but we suppress it as we are proving the bound for any large $m$.\vglue1mm
    
    Now for \(J_2\), if \(|t-s|\leq 2^{2j-5}r^2\) then using geometric considerations we have that  \(|x-y|\geq 2^{j-3}r\). We again observe that \(\lambda\mapsto \lambda^{-4m-3}e^{-\frac{t-s}{\lambda^2}}\) is maximized for \(\lambda=\sqrt{\frac{t-s}{4m+3}}\) and
    \(s\mapsto s^{-n-7/2}e^{-c\frac{2^{2j}r^2}{s}}\) is maximized for \(s=\frac{c}{n+7/2}2^{2j}r^2\). Using these observations, we have for $(z,\tau)\in \Delta_{2^jr}\setminus \Delta_{2^{j-1}r}$ 
    \begin{align*}
        &\iint_{\Gamma^{n}(z,\tau)}\frac{|J_2|^2}{\lambda^{n+3}}dx d\lambda dt
        \\
        &\lesssim \int_0^{\beta2^jr}\Big(\fint_{\Delta_{ca\lambda}(y,s)}\chi_{|x-y|\geq 2^{j-3}r}\frac{1}{\lambda^{4m+3}}(t-s)^{-n+2m-2}e^{-\frac{t-s}{\lambda^2}}e^{-c\frac{|x-y|^2}{t-s}}dxdt\Big) \Vert b\Vert_{L^1}^2 d\lambda
        \\
        &\lesssim \Vert b\Vert_{L^1}^2\int_0^{\beta2^jr}\Big(\fint_{\Delta_{ca\lambda}(z,\tau)}\frac{1}{(t-s)^{n+7/2}}e^{-c\frac{2^{2j}r^2}{t-s}}dxdt\Big)d\lambda
        \\
         &\lesssim \Vert b\Vert_{L^1}^2\int_0^{\beta 2^jr} (2^jr)^{-2n-7}d\lambda\lesssim 2^{-2(n+3)j}r^{-2(n+3)}\Vert b\Vert_{L^1}^2.
    \end{align*}
Here we have first optimized the second integral in $\lambda$ by replacing it with the supremum and then have found the supremum of the third integral in the time variable as well. \vglue1mm

For points $(X,t)\in \Gamma^{a}(z,\tau)$ we have $\lambda=x_{n+1}$ large, at least $\beta2^jr$. 
The second exponential in \eqref{eq:kernelBounds} is bounded by $1$. Optimising 
\eqref{eq:kernelBounds} in the time variable we see that \(\tau\mapsto \tau^{-\frac{n}{2}+m-1}e^{-\frac{\tau}{\lambda^2}}\) is maximized by \(\tau=(-\frac{n}{2}+m-1)\lambda^2\) if \(-\frac{n}{2}+m-1\geq 1\). But since we choose a large \(m\), this can be guaranteed. It then follows that
    \begin{multline*}
        \iint_{\Gamma^a(z,\tau)}\frac{|\lambda^{-1}\mathcal{P}_{\lambda}b|^2}{\lambda^{n+3}}dx d\lambda dt\\=\iint_{\Gamma^a(z,\tau)}\frac{\Vert b\Vert_{L^1}^2}{\lambda^{3n+9}}dx d\lambda dt
        \lesssim \int_{\beta2^jr}^\infty\frac{\Vert b\Vert_{L^1}^2}{\lambda^{2n+7}}dx d\lambda dt\lesssim 2^{-2(n+3)j}r^{-2(n+3)}\Vert b\Vert_{L^1}^2.
    \end{multline*}

    Hence we obtain in total (after adding contributions from $J_1$ and $J_2$):
    \begin{multline*}
        \Vert \mathcal{A}(\lambda^{-1}\mathcal{P}_{\lambda}b)\Vert_{L^1(\partial\Omega\setminus 8\Delta)}\lesssim \sum_{j\geq 4}\Vert \mathcal{A}(\lambda^{-1}\mathcal{P}_{\lambda}b)\Vert_{L^1(\Delta_{2^jr}\setminus \Delta_{2^{j-1}r})}
        \\
        \lesssim\sum_{j}(2^{j+1}r)^{n+2}2^{-(n+3)j}r^{-(n+3)}\Vert b\Vert_{L^1}
        \lesssim \frac{1}{r}\Vert b\Vert_{L^1}\lesssim C.
    \end{multline*}   
The last claim follows from the property (iii) of the $b$-atom. Hence indeed \eqref{eq-atom} holds.

\subsection{The area function $A(\lambda\nabla_\parallel\partial_\lambda \mathcal{P}_{\lambda}f)$}\label{s5.2}

The following lemma establishes a Caccioppoli inequality for the operator $(I+\lambda^2\mathcal H_\parallel)^{-1}$.

\begin{lemma}\label{L5.10} Fix any $\lambda>0$. Assume that $u:\R^n\times\R\to\R$ is a solution of the following parabolic PDE:
$$\partial_tu-\divg_x(A_\parallel(x,\lambda,t)\nabla_xu)+\frac{u}{\lambda^2}=\frac{g}{\lambda^2},$$
for $(x,t)\in \R^n\times\R$. That is, $u=(I+\lambda^2\mathcal H_\parallel)^{-1}g$.
Then for any parabolic ball $\Delta_\lambda=\Delta_\lambda(x,t)$ of radius $\lambda$ we have that
\begin{equation}\label{eq5.2}
\iint_{\Delta_\lambda} |\nabla_x u|^2 dx dt\lesssim \frac1{\lambda^2}\left[\iint_{2\Delta_\lambda} |u|^2dx dt+\iint_{2\Delta_\lambda}|g|^2dx dt\right].
\end{equation}
\end{lemma}

\begin{proof} We fix a ball $\Delta_\lambda$ and find a smooth cutoff function such that $\varphi=1$ on 
 $\Delta_\lambda$ and $\varphi$ vanishes outside $2\Delta_\lambda$. Given the scaling we may assume that
 $|\nabla_x\varphi|\lesssim \lambda^{-1}$ and  $|\partial_t\varphi|\lesssim \lambda^{-2}$. We multiply both sides of the PDE for u by $u\varphi^2$ and integrate by parts. This gives us:
\begin{multline*}
\frac12\iint_{\R^n\times \R}\left[\frac{d}{dt}(u\varphi)^2-u^2\varphi(\partial_t\varphi)+A_\parallel\nabla_xu\cdot (\nabla_x u\varphi^2)+\frac{(u\varphi)^2}{\lambda^2}\right]=\\\iint_{\R^n\times \R}\frac{(g\varphi)(u\varphi)}{\lambda^2}\le \frac12\iint_{2\Delta_\lambda}\frac{g^2}{\lambda^2}+\frac12\iint_{\R^n\times \R}\frac{(u\varphi^2)}{\lambda^2}.
\end{multline*}
The first term vanishes. After hiding the last term on the left-hand side we get that
$$
\iint_{\Delta_\lambda}|\nabla_x u|^2\lesssim \iint_{\R^n\times \R}A_\parallel(\nabla_xu\varphi)\cdot(\nabla_xu\varphi)\lesssim \frac12\iint_{2\Delta_\lambda}\frac{g^2}{\lambda^2}+C\iint_{2\Delta_\lambda}\frac{u^2}{\lambda^2},
$$
where we have used Cauchy-Schwarz and the arithmetic-geometric inequality multiple times and the estimates for $\nabla_x\varphi$ and $\partial_t\varphi$. This establishes our claim. 
\end{proof}

As a corollary we obtain the following. For any point $(y,s)\in \partial\Omega$ consider the nontangential cone $\Gamma_a(y,s)$, where $a>0$ is the aperture of the cone.

 For each $\lambda>0$ consider the set 
 \begin{equation}\label{eq-GL}
 \Gamma_a^\lambda(y,s):=\Gamma_a(y,s)\cap \{(x,\lambda,t): (x,t)\in \R^n\times\R\}
 \end{equation}
This set has parabolic diameter $\sim a\lambda$ and hence can be covered by at most $k=k(n,a)<\infty$
parabolic balls on $\R^n\times\R$ of radius $\lambda$. It follows that if we write using \eqref{P-lambda}
\begin{multline*}
u=\partial_\lambda(I+\lambda^2\mathcal H_\parallel)^{-m}f=\frac{m}{\lambda}\left[2\mathcal P_{\lambda,m+1}f-2\mathcal P_{\lambda,m}f\right]=(I+\lambda^2\mathcal H_\parallel)^{-1}
\frac{m}{\lambda}\left[2\mathcal P_{\lambda,m}f-2\mathcal P_{\lambda,m-1}f\right]\\
=(I+\lambda^2\mathcal H_\parallel)^{-1}\frac{m}{m-1}\partial_\lambda(I+\lambda^2\mathcal H_\parallel)^{-m+1}f:=
(I+\lambda^2\mathcal H_\parallel)^{-1}g,
\end{multline*}
for $g:=\frac{m}{m-1}\partial_\lambda(I+\lambda^2\mathcal H_\parallel)^{-m+1}f$,
then by \eqref{eq5.2}
\begin{multline*}
\lambda^2\iint_{\Delta_\lambda} |\nabla_\parallel \partial_\lambda(I+\lambda^2\mathcal H_\parallel)^{-m}f|^2 dx dt\lesssim \iint_{2\Delta_\lambda} |\partial_\lambda(I+\lambda^2\mathcal H_\parallel)^{-m}f|^2dx dt+\\\iint_{2\Delta_\lambda}|\partial_\lambda(I+\lambda^2\mathcal H_\parallel)^{-m+1}f|^2dx dt.
\end{multline*}
We now sum these Caccioppoli-type inequalities over the at most $k=k(n,a)$ parabolic balls $\Delta_\lambda$ covering $\Gamma_a^\lambda(y,s)$. Each ball appears at most a fixed bounded number of times in the doubled covering $\{2\Delta_\lambda\}$, and the union $\bigcup 2\Delta_\lambda$ is contained in a single slice $\Gamma_b^\lambda(y,s)$ for some aperture $b=b(a,n)>a$. (This is the elementary geometric step: doubling a Whitney-type covering at height~$\lambda$ enlarges only the aperture of the enclosing cone, not its height.) Therefore
\begin{multline*}
\iint_{\Gamma^\lambda_a(y,s)} |\lambda\nabla_\parallel \partial_\lambda(I+\lambda^2\mathcal H_\parallel)^{-m}f|^2\, dx dt\\
\lesssim \iint_{\Gamma^\lambda_b(y,s)} \left[|\partial_\lambda(I+\lambda^2\mathcal H_\parallel)^{-m}f|^2+|\partial_\lambda(I+\lambda^2\mathcal H_\parallel)^{-m+1}f|^2\right]dx dt,
\end{multline*}
where the implicit constant depends only on $n$ and on the apertures $a,b$.
We now multiply both sides by $\lambda^{-n-3}$, integrate in $\lambda$ over $(0,\infty)$ and take square root. It follows that
\begin{multline*}
\mathcal A_a(\lambda \nabla_\parallel\partial_\lambda\mathcal P_{\lambda,m}f)(y,s)=\left(\iint_{\Gamma_a(y,s)} |\lambda\nabla_\parallel \partial_\lambda(I+\lambda^2\mathcal H_\parallel)^{-m}f|^2\lambda^{-n-3}\, dx\, d\lambda\, dt\right)^{1/2}\\
\lesssim \left(\iint_{\Gamma_b(y,s)} | \partial_\lambda(I+\lambda^2\mathcal H_\parallel)^{-m}f|^2\lambda^{-n-3}\, dx\, d\lambda\, dt\right)^{1/2}+ \left(\iint_{\Gamma_b(y,s)} | \partial_\lambda(I+\lambda^2\mathcal H_\parallel)^{-m+1}f|^2\lambda^{-n-3}\, dx\, d\lambda\, dt\right)^{1/2}\\
=\left[\mathcal A_b(\partial_\lambda\mathcal P_{\lambda,m}f)+\mathcal A_b(\partial_\lambda\mathcal P_{\lambda,m-1}f)\right](y,s).
\end{multline*}
It follows that all $L^p$ bounds we have established in the previous subsection for $\mathcal A(\partial_\lambda\mathcal P_{\lambda,m}f)$, where $m$ is large, must also hold for $\mathcal A(\lambda \nabla_\parallel\partial_\lambda\mathcal P_{\lambda,m}f)$.

\subsection{The remaining area functions of $\mathcal P_\lambda f$}

In this section we deduce the remaining bounds from part 1) of \eqref{list} using the estimates we have established in \ref{s5.1}-\ref{s5.2}.

Let us recall the calculation \eqref{P-lambda22} from the previous section. A similar argument shows that
\begin{equation}\label{P-lambda}
\partial_\lambda \mathcal{P}_{\lambda,m}=\frac{2m}{\lambda}\left[\mathcal P_{\lambda,m+1}-\mathcal P_{\lambda,m}\right]
\quad\mbox{and}\quad \mathcal H_\parallel \mathcal{P}_{\lambda,m}=\frac1{\lambda^2}\left[\mathcal{P}_{\lambda,m-1}-\mathcal{P}_{\lambda,m}\right].
\end{equation}
Therefore
$$\lambda\mathcal H_\parallel \mathcal{P}_{\lambda,m}=-\frac1{2m+2}\partial_\lambda\mathcal P_{\lambda,m+1}.$$
and hence the second estimate in the first line of \eqref{list} follows from results of subsection \ref{s5.1}.
Similarly,
\begin{equation}\label{eq513}
\lambda\partial^2_\lambda \mathcal{P}_{\lambda,m}=2m\partial_\lambda\mathcal P_{\lambda,m+1}-2m\partial_\lambda\mathcal P_{\lambda,m}
\end{equation}
and hence again the estimate for the first term of the second line of \eqref{list} again follows. For the second term  of the second line of \eqref{list} we have
\begin{equation}\label{eq514}
\lambda^2\nabla_\parallel\partial^2_\lambda \mathcal{P}_{\lambda,m}=2m\lambda\nabla_\parallel\partial_\lambda\mathcal P_{\lambda,m+1}-2m\lambda\nabla_\parallel\partial_\lambda\mathcal P_{\lambda,m}
\end{equation}
and therefore bounds for this term follow from the bounds of $\mathcal A(\lambda\nabla_\parallel\partial_\lambda \mathcal{P}_{\lambda}f)$ which were proven in subsection \ref{s5.2}.

To summarize the whole section, we have the following:

\begin{lemma}\label{l-area} Let $1<p\le 2$. For all $m>0$ sufficiently large we have
\begin{multline}\label{e52}
\Vert \mathcal{A}(\partial_{\lambda}\mathcal{P}_{\lambda}f)\Vert_{L^p}+\|\mathcal{A}(\lambda\mathcal{H}_{||}\mathcal{P}_{\lambda}f)\Vert_{L^p}+\|\mathcal A(\lambda\nabla_\parallel\partial_\lambda \mathcal{P}_{\lambda}f)\|_{L^p}\\
+\Vert \mathcal{A}(\lambda\partial^2_{\lambda}\mathcal{P}_{\lambda}f)\Vert_{L^p}+\Vert \mathcal{A}(\lambda^2\nabla_{||}\partial^2_{\lambda}\mathcal{P}_{\lambda}f)\Vert_{L^p}\lesssim \|f\|_{\dot{L}^p_{1,1/2}}.
\end{multline}
\end{lemma}

\section{Nontangential maximal function estimates}\label{SN}

The main goal of the section is to establish bounds under point 2) of \eqref{list}. Let us mention that for some of these bounds we use the $\sup$ version of the nontangential maximal function which is denoted as $N$, whereas for others we use the $L^2$-averaged version of this function as defined in \eqref{def.Nap}-\eqref{def.N2} and which is denoted by $\tilde N$. We refer the reader to Definition \ref{def-N} as this distinction matters and not all nontangential bounds we establish are of the strong type using $N$.

We keep the same notation as in the previous section. As currently defined, $\mathcal P_{\lambda,m}$ acts on the space $\dot{L}^2_{1,1/2}(\R^n\times\R)$. However, since for a constant $c$ we have $\mathcal H_\parallel c=0$ (because the tangential derivatives of $c$ all vanish), then also
 $(I+\lambda^2\mathcal H_\parallel)c=c$ and hence we may extend the definition of the resolvent $(I+\lambda^2\mathcal H_\parallel)^{-1}$ by adding constants to the space
$\dot{L}^2_{1,1/2}(\R^n\times\R)$. That is, $\mathcal P_{\lambda,m}$ can be thought of as acting on the functions
\begin{equation}
{\E}(\R^n\times \R):=\{c\in\R\}+\dot{L}^2_{1,1/2}(\R^n\times\R),
\end{equation}
with norm 
$$\left|\fint_{\Delta_1(0,0)}f\,dx dt\right|+\|\nabla_x f\|_{L^2}+\|D^{1/2}_t f\|_{L^2}.$$

\subsection{The nontangential maximal function $ N(\partial_\lambda\mathcal P_\lambda f)$ and $ \tilde N(\nabla_\parallel\mathcal P_\lambda f)$}\label{s6.1}
$\mbox{ }$\medskip

Lemma \ref{L5.10} is again relevant here. Fix $m>0$ and $\lambda>0$ and consider 
$u=\mathcal P_{\lambda,m}( f-\langle f\rangle_{\Delta_\lambda})$ for some fixed boundary parabolic ball 
$\Delta_\lambda$ of radius $\lambda$. Here $\langle f\rangle_{\Delta_\lambda}$ denotes the average over $\Delta_\lambda$.
Since $\nabla_\parallel$ kills constants $\nabla_\parallel \mathcal P_{\lambda,m}( f-\langle f\rangle_{\Delta_\lambda})=\nabla_\parallel \mathcal P_{\lambda,m} f$. Writing $u$ as $(I+\lambda^2\mathcal H_\parallel)^{-1}\mathcal P_{\lambda,m-1}( f-\langle f\rangle_{\Delta_\lambda})$ we have by 
Lemma \ref{L5.10}
\begin{multline*}
\iint_{\Delta_\lambda} |\nabla_\parallel (I+\lambda^2\mathcal H_\parallel)^{-m}f|^2 dx dt\lesssim \iint_{2\Delta_\lambda}|\lambda^{-1}\mathcal P_{\lambda,m}( f-\langle f\rangle_{\Delta_\lambda})|^2dx dt+\\\iint_{2\Delta_\lambda}|\lambda^{-1}\mathcal P_{\lambda,m-1}( f-\langle f\rangle_{\Delta_\lambda})|^2dx dt.
\end{multline*}
It follows that if we can understand 
${N}(\lambda^{-1}\mathcal P_{\lambda}( f-\langle f\rangle_{\Delta_\lambda}))$, then the bounds for 
$ \tilde N(\nabla_\parallel\mathcal P_\lambda f)$ would follow. \vglue1mm

Similarly, for $\partial_\lambda\mathcal P_{\lambda,m} f$ we know (thanks to \eqref{P-lambda}) that it can be written  as a difference $\frac{2m}{\lambda}\left[\mathcal P_{\lambda,m+1}-\mathcal P_{\lambda,m}\right]$,
which kills constants since 
$$\left[\mathcal P_{\lambda,m+1}-\mathcal P_{\lambda,m}\right]c=c-c=0.$$
It follows that 
$$\partial_\lambda\mathcal P_{\lambda,m} f=
\frac{2m}{\lambda}\left[\mathcal P_{\lambda,m+1}-\mathcal P_{\lambda,m}\right]f=
\frac{2m}{\lambda}\left[\mathcal P_{\lambda,m+1}-\mathcal P_{\lambda,m}\right](f-\langle f\rangle_{\Delta_\lambda}).$$
Once again, it suffices to understand ${N}(\lambda^{-1}\mathcal P_{\lambda}( f-\langle f\rangle_{\Delta_\lambda}))$. 
In summary, it suffices in both cases to find good nontangential bounds of $\lambda^{-1}\mathcal P_{\lambda}( f-\langle f\rangle_{\Delta_\lambda})$. 

Now, the notation here is a bit unfortunate and invites confusion if we are not careful enough. To be more precise, by 
nontangential bounds for ${N}(\lambda^{-1}\mathcal P_{\lambda}( f-\langle f\rangle_{\Delta_\lambda}))$ we will understand that for a fixed boundary point $(z,\tau)\in\R^n\times\R$ and its nontangential cone $\Gamma(z,\tau)$ we will only consider balls $\Delta_\lambda=\Delta_\lambda(z,\tau)$ with fixed center (given by the vertex of the nontangential cone). This implies that, within the cone $\Gamma(z,\tau)$, the average $\langle f\rangle_{\Delta_\lambda}$ is a constant function (only depending on $\lambda$). We now have an unambiguous definition of the expression ${N}(\lambda^{-1}\mathcal P_{\lambda}( f-\langle f\rangle_{\Delta_\lambda}))(z,\tau)$.\vglue2mm

Let us fix a point $(z,\tau)\in\R^n\times\R$ and its nontangential cone $\Gamma(z,\tau)$.
Consider any $\lambda>0$ and any point $(X,t)=(x,\lambda,t)\in \Gamma(z,\tau)$. We decompose
the boundary $\R^n\times \R$ as follows:

Let $\Delta_{2^j\lambda}=\Delta_{2^j\lambda}(z,\tau)$ be boundary balls of radius $2^j\lambda$ centered at
$(z,\tau)$ (the vertex of $\Gamma=\Gamma_a(z,\tau)$). Let $j_0$ be the smallest positive integer such that for any $(x,\lambda,t)\in \Gamma_a(z,\tau)$ and $(y,s)\notin \Delta_{2^{j_0}\lambda}$ we have separation
of parabolic distances of $(x,t)$ and $(y,s)$ by at least $\lambda$, i.e.,
 $$d_p((x,t),(y,s))=\|(x-y,t-s)\| \ge \lambda,$$
where $\|\cdot\|$ is defined by \eqref{E:par-norm}. Such a $j_0$ always exists (only depending on $a$ - the aperture  of $\Gamma$). We then write

\begin{equation}\label{eq6.2a}
\R^n\times\R=\Delta_{2^{j_0}\lambda}\cup \bigcup_{j> j_0}\left(\Delta_{2^j\lambda}\setminus \Delta_{2^{j-1}\lambda}\right).
\end{equation}
The dyadic nature of this decomposition ensures that there exists a fixed $\beta>0$ such that
\begin{multline}\label{eq6.2}
\mbox{If }(x,\lambda,t)\in\Gamma(z,\tau)\mbox{ and }(y,s)\in \left(\Delta_{2^j\lambda}\setminus \Delta_{2^{j-1}\lambda}\right)\\\mbox{ then } d_p((x,t),(y,s))=\|(x-y,t-s)\| \ge \beta2^{j}\lambda \quad\mbox{for all }j>j_0.
\end{multline}

Using the kernel bounds \eqref{eq:kernelBounds} we may write an estimate for the value of $\lambda^{-1}\mathcal P_{\lambda}(f-\langle f\rangle_{\Delta_\lambda})(x,\lambda,t)$ as follows:

\begin{multline}\label{eq6.3}
|\lambda^{-1}\mathcal P_{\lambda}( f-\langle f\rangle_{\Delta_\lambda})(x,\lambda,t)|\le 
\int_{\Delta_{2^{j_0}\lambda}}|\lambda^{-1}K_{\lambda,m}(x,t,y,s)||f-\langle f\rangle_{\Delta_\lambda}|(y,s)dyds\\
+\sum_{j>j_0}\int_{\Delta_{2^{j}\lambda}\setminus \Delta_{2^{j-1}\lambda}}|\lambda^{-1}K_{\lambda,m}(x,t,y,s)||f-\langle f\rangle_{\Delta_\lambda}|(y,s)dyds
\\
=\int_{\Delta_{2^{j_0}\lambda}}|\lambda^{-1}K_{\lambda,m}(x,t,y,s)||f-\langle f\rangle_{\Delta_{2^{j_0}\lambda}}|(y,s)dyds
\\
+\int_{\Delta_{2^{j_0}\lambda}}|\lambda^{-1}K_{\lambda,m}(x,t,y,s)|(\langle f\rangle_{\Delta_{2^{j_0}\lambda}}-\langle f\rangle_{\Delta_\lambda}|(y,s)dyds
\\
+\sum_{j>j_0}\int_{\Delta_{2^{j}\lambda}\setminus \Delta_{2^{j-1}\lambda}}|\lambda^{-1}K_{\lambda,m}(x,t,y,s)||f-\langle f\rangle_{\Delta_{2^j\lambda}}|(y,s)dyds
\\
+\sum_{j>j_0}\int_{\Delta_{2^{j}\lambda}\setminus \Delta_{2^{j-1}\lambda}}|\lambda^{-1}K_{\lambda,m}(x,t,y,s)||\langle f\rangle_{\Delta_{2^j\lambda}}-\langle f\rangle_{\Delta_\lambda}|(y,s)dyds.
\end{multline}
Hence there are four terms to estimate. Given \eqref{eq:kernelBounds} and \eqref{eq6.2} for sufficiently large $m>n/2+1$
we have
\begin{multline}\label{eq6.4}
|\lambda^{-1}K_{\lambda,m}(x,t,y,s)|\le \frac{C}{\lambda^{n+3}}\mbox{ for }(y,s)\in \Delta_{2^{j_0}\lambda}\qquad\mbox{and}\\
|\lambda^{-1}K_{\lambda,m}(x,t,y,s)|\le \frac{C}{\lambda^{n+3}}e^{-c4^j} \mbox{ for }(y,s)\in \Delta_{2^{j}\lambda}\setminus \Delta_{2^{j-1}\lambda},\quad j>j_0.
\end{multline}
Let us also recall the definition of the sharp maximal function $f^\sharp$ 
$$f^\sharp(x,t)=\sup_{\Delta\ni (x,t)}\, r(\Delta)^{-1}\fint_{\Delta}|f-\langle f\rangle_{\Delta}|.$$
Here $r(\Delta)$ denotes the radius of a parabolic ball $\Delta$. Then 
$$f^\sharp(x,t)\lesssim M(|\nabla_x f|)(x,t)+M_xM_t(|H_tD^{1/2}_t f|)(x,t),$$
and for any parabolic ball $\Delta$ centered at $(x,t)$
\begin{equation}\label{eq-poinc}
\frac1{r(\Delta)}\, \fint_{\Delta}\left|f-\langle f\rangle_{\Delta}\right|dxdt\le f^\sharp(x,t).
\end{equation}
See for example \cite{DinN}. It follows that for $f\in \dot{L}^p_{1,1/2}(\R^n\times\R)$ and $1<p<\infty$ 
$f^\sharp$  belongs to $L^p(\R^n\times\R)$ and its norm is controlled by $\|f\|_{\dot{L}^p_{1,1/2}(\R^n\times\R)}$. In particular this implies that
\begin{equation}\label{eq6.5}
\fint_{\Delta_{2^j\lambda}}|f(y,s)-\langle f\rangle_{\Delta_{2^j\lambda}}|\le 2^{j}\lambda f^\sharp(z,\tau),\quad |\langle f\rangle_{\Delta_{2^j\lambda}}-\langle f\rangle_{\Delta_\lambda}|\le Cj2^{j}\lambda f^\sharp(z,\tau).
\end{equation}
Using \eqref{eq6.4}-\eqref{eq6.5} in \eqref{eq6.3} we then have

\begin{multline}\label{eq6.6}
|\lambda^{-1}\mathcal P_{\lambda}( f-\langle f\rangle_{\Delta_\lambda})(x,\lambda,t)|\\
\lesssim f^\sharp(z,\tau)\int_{\Delta_{2^{j_0}\lambda}}\lambda^{-n-2} dyds +
f^\sharp(z,\tau)\int_{\Delta_{2^{j_0}\lambda}}\lambda^{-n-2} dyds
\\
+f^\sharp(z,\tau)\sum_{j>j_0}\int_{\Delta_{2^{j}\lambda}\setminus \Delta_{2^{j-1}\lambda}}\lambda^{-n-2}2^{j}e^{-c4^j}dyds
+f^\sharp(z,\tau)\sum_{j>j_0}j\int_{\Delta_{2^{j}\lambda}\setminus \Delta_{2^{j-1}\lambda}}\lambda^{-n-2}2^je^{-c4^j}dyds.
\end{multline}
We may now enlarge the integrals in the last line and integrate over $\Delta_{2^{j}\lambda}$ instead of the annuli. Given that $|\Delta_{2^{j}\lambda}|\sim 2^{(n+2)j}\lambda^{n+2}$ it follows that
\begin{multline*}
|\lambda^{-1}\mathcal P_{\lambda}( f-\langle f\rangle_{\Delta_\lambda})(x,\lambda,t)|\lesssim
M(f^\sharp)(z,\tau)\\+\sum_{j > j_0}j2^{(n+3)j}e^{-c4^j}\left[M(f^\sharp)(z,\tau)+f^\sharp(z,\tau)\right]\lesssim M(f^\sharp)(z,\tau).
\end{multline*}
and therefore also
$$N(\lambda^{-1}\mathcal P_{\lambda}( f-\langle f\rangle_{\Delta_\lambda}))(z,\tau)=\sup_{(x,\lambda,t)\in\Gamma(z,\tau)}|\lambda^{-1}\mathcal P_{\lambda}( f-\langle f\rangle_{\Delta_\lambda})(x,\lambda,t)|\lesssim  M(f^\sharp)(z,\tau),
$$
and since $M(f^\sharp)\in L^p(\R^n\times\R)$ whenever $f\in \dot{L}^p_{1,1/2}(\R^n\times\R)$ and $1<p<\infty$, we have just shown that
$$ \|N(\partial_\lambda\mathcal P_\lambda f)\|_{L^p(\R^n\times\R)}+ \|\tilde N(\nabla_\parallel\mathcal P_\lambda f)\|_{L^p(\R^n\times\R)}\le C
\|f\|_{\dot{L}^p_{1,1/2}(\R^n\times\R)}.$$

\subsection{The remaining nontangential maximal functions of $\mathcal P_\lambda f$}

Clearly, \eqref{eq513} immediately implies that the bounds we have established in subsection \ref{s6.1}
imply nontangential bounds for $\lambda\partial^2_\lambda\mathcal P_\lambda f$. By \eqref{P-lambda}
we see that
\begin{equation}\nonumber
\lambda\nabla_\parallel\partial_\lambda \mathcal{P}_{\lambda,m}=2m\nabla_\parallel\mathcal P_{\lambda,m+1}-2m\nabla_\parallel\mathcal P_{\lambda,m},
\end{equation}
from which we conclude that the nontangential bounds for $\lambda\nabla_\parallel\partial_\lambda \mathcal{P}_{\lambda}$ follows from the result of subsection \ref{s6.1}.
In summary,

\begin{lemma}\label{l-nont} Let $1<p<\infty$. For all $m>0$ sufficiently large we have
\begin{multline}\label{e62}
\Vert {N}(\partial_{\lambda}\mathcal{P}_{\lambda}f)\Vert_{L^p}+
\Vert \tilde{N}(\lambda\nabla_{||}\partial_{\lambda}\mathcal{P}_{\lambda}f)\Vert_{L^p}
+\Vert {N}(\lambda\partial^2_{\lambda}\mathcal{P}_{\lambda}f)\Vert_{L^p}+\Vert \tilde{N}(\nabla_{||}\mathcal{P}_{\lambda}f)\Vert_{L^p}\lesssim \|f\|_{\dot{L}^p_{1,1/2}}.
\end{multline}
\end{lemma}

\section{Carleson measure bound}

The part 3) of \eqref{list} will follow from the following lemma:

\begin{lemma}\label{lemma:CarelsonBound}
    Let \(g:\mathbb{R}^{n}\times\R\to\mathbb{R}\) be a bounded function. Then
    \begin{align}\Vert C(\lambda \mathcal{P}_{\lambda}\partial_jg)\Vert_{L^\infty}\lesssim \Vert g \Vert_{L^\infty}.\label{eq:CarlesonFctBound}\end{align}
\end{lemma}
Although the operator in \eqref{list} is $\tilde{\mathcal P}^*_\lambda$, it suffices to prove it for  
$\mathcal{P}_{\lambda}$ as these operators are analogous.

    Let us denote the resolvent by \(\mathcal{E}_\lambda:=(I+\lambda^2\mathcal{H}_{||})^{-1}\) and rewrite the operator \(\mathcal{P}_\lambda=\mathcal{P}_{\lambda,m}=\mathcal{E}_\lambda^m\).
    
The proof of Lemma \ref{lemma:CarelsonBound} proceeds by induction starting from Lemma 4.4 of \cite{AEN2}, which establishes that the operator  \(\mathcal{E}_\lambda\)
enjoys good off-diagonal estimates:
    
\begin{lemma}\label{lemma:OffDiagonalResolvent}
    Let \(E,F\subset\R^n\times\R\) two measurable sets with their parabolic distance \(d:=\mathrm{dist}(E,F)>0\). Then there exists \(c>0\) such that
    \begin{enumerate}
        \item \(\displaystyle\int_F |\mathcal{E}_\lambda f|^2 + |\lambda \nabla_{||}\mathcal{E}_\lambda f|^2 dxdt\lesssim e^{-c\frac{d}{\lambda}}\int_E |f|^2dxdt, \) and
        \item \(\displaystyle\int_F |\lambda\mathcal{E}_\lambda \mathrm{div}_\parallel(g)|^2 dxdt\lesssim e^{-c\frac{d}{\lambda}}\int_E |g|^2dxdt \)
    \end{enumerate}
    for all \(f\in L^2(\mathbb{R}^{n+1}), g\in L^2(\mathbb{R}^{n+1})^n\) supported on \(E\).
\end{lemma}    

We begin the proof of Lemma \ref{lemma:CarelsonBound}.
\begin{proof}    
We use induction in $m$. When $m=1$ the estimate \eqref{eq:CarlesonFctBound} claims that 
$$\Vert C(\lambda \mathcal{E}_{\lambda}\partial_jg)\Vert_{L^\infty}\lesssim \Vert g \Vert_{L^\infty},$$
which was proven in Lemma 8.3 of \cite{AEN2} for the special case of weight $w=1$.\vglue1mm

Assume now that we have already established the estimate for $m-1$, i.e.,
$$\Vert C(\lambda \mathcal{E}^{m-1}_{\lambda}\partial_jg)\Vert_{L^\infty}\lesssim \Vert g \Vert_{L^\infty}.$$
We shall establish this estimate for $m$.\vglue1mm

Fix an arbitrary boundary ball \(\Delta=\Delta(y,s)\subset \R^n\times\R\).  Since it is more convenient to use rectangular Carleson regions we set $T(\Delta)=\Delta\times (0,r(\Delta))$, where $r=r(\Delta)$ is the radius of the ball.   Then
 \begin{multline*}
        \frac{1}{\sigma(\Delta)}\iint_{T(\Delta)}\frac{|\lambda\mathcal{P}_{\lambda,m}\partial_j g|^2}{\lambda}dxdtd\lambda=\frac{1}{\sigma(\Delta)}\iint_{T(\Delta)}\frac{|\mathcal{E}_\lambda(\lambda\mathcal{E}_\lambda^{m-1}\partial_j g)|^2}{\lambda}dxdtd\lambda
        \\
        =\frac{1}{\sigma(\Delta)}\int_0^{r(\Delta)}\lambda\left[\int_\Delta |\mathcal{E}_\lambda(\chi_{\Delta_{2r}}\mathcal{E}_\lambda^{m-1}\partial_j g)|^2dxdt+ \int_{\Delta}|\sum_{j\geq 1}\mathcal{E}_\lambda(\chi_{\Delta_{2^{j+1}r}\setminus\Delta_{2^jr}}\mathcal{E}_\lambda^{m-1}\partial_j g)|^2 dxdt\right]d\lambda
        \\
        =:\frac{1}{\sigma(\Delta)}\int_0^{r(\Delta)} \lambda(I_\lambda+II_\lambda) d\lambda.
    \end{multline*}
    Here $\Delta_{2^{j+1}r}$, $j=1,2,\dots$, are concentric parabolic balls with the same center as $\Delta$.
    
    For the first term $I_\lambda$ we use that \(\mathcal{E}_\lambda\) is bounded on $L^2$.

    This can be seen easily. For $u=\mathcal E_\lambda g$ we need to solve the PDE
    $$\partial_tu-\divg_\parallel(A_\parallel(x,\lambda,t)\nabla_\parallel u)+\frac{u}{\lambda^2}=\frac{g}{\lambda^2},$$
    for which after multiplying by $u$ and integrating by parts gives  us that
\begin{equation}\label{E-bound}
\|\lambda\nabla_\parallel\mathcal E_\lambda g\|^2_{L^2} +   \|\mathcal E_\lambda g\|^2_{L^2}\lesssim
\int_{\R^n\times\R} \left[\lambda^2A_\parallel\nabla_\parallel u\cdot\nabla_\parallel u+u^2\right]dxdt \lesssim C\|g\|^2_{L^2}.
\end{equation}
It follows that    
    \[I_\lambda\lesssim \int_{2\Delta} |\mathcal{E}_\lambda^{m-1}\partial_j g|^2dxdt,\]
    and therefore
    \[\frac{1}{\sigma(\Delta)}\int_0^{r(\Delta)} \lambda I_\lambda d\lambda\lesssim \frac{1}{\sigma(\Delta)}\iint_{T(2\Delta)} \frac{|\lambda\mathcal{E}_\lambda^{m-1}\partial_j g|^2}{\lambda}dxdtd\lambda\lesssim C(|\lambda\mathcal{E}_\lambda^{m-1}\partial_j g|)^2(y,s),\]
which is bounded by $\|g\|^2_{L^\infty}$ by the induction assumption.\vglue1mm

For the second term $II_\lambda$ we can use Minkowski's inequality, 
 the off-diagonal estimate of Lemma \ref{lemma:OffDiagonalResolvent} and H\"older's inequality to obtain
 \begin{multline*}
        (II)_\lambda^{1/2}=\Big(\int_{\Delta}\big|\sum_{j\geq 1}\mathcal{E}_\lambda(\chi_{\Delta_{2^{j+1}r}\setminus\Delta_{2^{j}r}}\mathcal{E}_\lambda^{m-1}\partial_j g)\big|^2dxdt\Big)^{1/2}
        \\
        \le\sum_{j\geq 1}\Big(\int_{\Delta}|\mathcal{E}_\lambda(\chi_{\Delta_{2^{j+1}r}\setminus\Delta_{2^{j}r}}\mathcal{E}_\lambda^{m-1}\partial_j g)|^2dxdt\Big)^{1/2}
        \\
        \lesssim\sum_{j\geq 1}e^{-c\frac{2^{j}r(\Delta)}{\lambda}}\Big(\int_{\Delta_{2^{j+1}r}\setminus\Delta_{2^{j}r}}|\mathcal{E}_\lambda^{m-1}\partial_j g|^2dxdt\Big)^{1/2}
        \\
        \lesssim\Big(\sum_{j\geq 1}e^{-c\frac{2^{j}r(\Delta)}{\lambda}}\Big)^{1/2}\Big(\sum_{j\geq 1}e^{-c\frac{2^{j}r(\Delta)}{\lambda}}\int_{\Delta_{2^{j+1}r}\setminus\Delta_{2^{j}r}}|\mathcal{E}_\lambda^{m-1}\partial_j g|^2dxdt\Big)^{1/2}.
    \end{multline*}
    Since \(\lambda\in (0,r(\Delta))\), we have \(e^{-c\frac{2^{j}r(\Delta)}{\lambda}}\leq e^{-c2^j}\) and \(\sum_{j\geq 1}e^{-c\frac{2^{j}r(\Delta)}{\lambda}}\) is uniformly bounded independent of \(\lambda\) and \(\Delta\). Thus, we can continue with estimating
    \begin{multline*}
        \frac{1}{\sigma(\Delta)}\int_0^{r(\Delta)} \lambda II_\lambda d\lambda
        \lesssim \frac{1}{\sigma(\Delta)}\sum_{j\geq 1}e^{-c2^{j}}\iint_{(\Delta_{2^{j+1}r}\setminus \Delta_{2^{j}r})\times(0,r(\Delta))} \frac{|\lambda\mathcal{E}_\lambda^{m-1}\partial_j g|^2}{\lambda}dxdtd\lambda
        \\
        \lesssim \sum_{j\geq 1}2^{(j+1)(n+2)}e^{-c2^{j}}\frac{1}{\sigma(\Delta_{2^{j+1}r})}\iint_{T(\Delta_{2^{j+1}r})} \frac{|\lambda\mathcal{E}_\lambda^{m-1}\partial_j g|^2}{\lambda}dxdtd\lambda
        \\
        \lesssim \sum_{j\geq 1}2^{(j+1)(n+2)}e^{-c2^{j}}C(\lambda\mathcal{E}_\lambda^{m-1}\partial_j g)(y,s)
        \lesssim C(\lambda\mathcal{E}_\lambda^{m-1}\partial_j g)^2(y,s),
    \end{multline*}
which is bounded by $\|g\|^2_{L^\infty}$, again by the induction assumption. After taking supremum over all boundary balls $\Delta$ this then yields
 \begin{align*}
        \Vert C(\lambda\mathcal{P}_{\lambda,{m}}\partial_j g)\Vert_{L^\infty}= \Vert C(\lambda\mathcal{E}_\lambda^{m}\partial_j g)\Vert_{L^\infty}\lesssim\Vert C(\lambda\mathcal{E}_\lambda^{m-1}\partial_j g)\Vert_{L^\infty}\lesssim \Vert g\Vert_{L^\infty}.
    \end{align*} 
\end{proof}

We also note that the off-diagonal estimates of Lemma \ref{lemma:OffDiagonalResolvent} hold for $\mathcal P_{\lambda,m}$ for all $m\ge 1$. Observations of this sort have been made in various contexts (see, for example, Lemma 4.6 of \cite{AE}); for convenience, we state and prove it below.

\begin{lemma}\label{lemma:OffDiagonalResolvent2}
    Let \(E,F\subset\R^n\times\R\) two measurable sets with their parabolic distance \(d:=\mathrm{dist}(E,F)>0\). Then there exists \(C>0\) such that
    \begin{enumerate}
        \item \(\displaystyle\int_F |\mathcal{P}_\lambda f|^2 + |\lambda \nabla_{||}\mathcal{P}_\lambda f|^2 dxdt\lesssim e^{-c\frac{d}{\lambda}}\int_E |f|^2dxdt, \) and
        \item \(\displaystyle\int_F |\lambda\mathcal{P}_\lambda \mathrm{div}_\parallel(g)|^2 dxdt\lesssim e^{-c\frac{d}{\lambda}}\int_E |g|^2dxdt \)
    \end{enumerate}
    for all \(f\in L^2(\mathbb{R}^{n+1}), g\in L^2(\mathbb{R}^{n+1})^n\) supported on \(E\).
\end{lemma}    

\begin{proof} Observe that in the regime $d\lesssim \lambda$ the right-hand side of (1) contains no decay factor and the claim reduces to a uniform $L^2$ bound. The estimate \eqref{E-bound} provides this bound for the operator $\mathcal E_\lambda$ as well as for $\lambda\nabla_\parallel \mathcal E_\lambda$. Writing
$$\mathcal P_\lambda f=\mathcal P_{\lambda,m}f=\mathcal E_\lambda \mathcal P_{\lambda,m-1}f=\mathcal E_\lambda g,\qquad
g:=\mathcal P_{\lambda,m-1}f,$$
applying \eqref{E-bound} once yields
\begin{equation}\label{eqsas}
\|\mathcal P_{\lambda,m}f\|_{L^2(F)}+\|\lambda\nabla_\parallel\mathcal P_{\lambda,m}f\|_{L^2(F)}\lesssim \|\mathcal P_{\lambda,m-1}f\|_{L^2}.
\end{equation}
Iterating this peels off one factor of $\mathcal E_\lambda$ at a time, i.e., for each $k$:
$\|\mathcal P_{\lambda,m-k}f\|_{L^2}\lesssim \|\mathcal P_{\lambda,m-k-1}f\|_{L^2}$,
and hence by induction $\|\mathcal P_{\lambda,m-1}f\|_{L^2}$ is eventually controlled by $\|f\|_{L^2(E)}$. In particular, both summands $|\mathcal P_\lambda f|^2$ and $|\lambda\nabla_\parallel \mathcal P_\lambda f|^2$ appearing in claim~(1) are controlled by $\|f\|_{L^2(E)}$ thanks to \eqref{eqsas}.

The estimate (2) follows from (1) by duality, hence it suffices to prove (1) in the case $d\gg\lambda$.
 For the second summand in the integral involving $\lambda\nabla_\parallel \mathcal P_\lambda f$ we observe that Lemma \ref{L5.10} allows to remove $\lambda\nabla_\parallel$
from the operator at the expense of slightly enlarging the set, which then reduces the claim to the case of 
$\mathcal P_\lambda f$. Thus it suffices to prove the part of statement (1) for the first integrand.\vglue1mm

The proof proceeds by induction on $m$, and when $m=1$ there is nothing left to prove. 

Assume now that our lemma already holds for $\mathcal P_\lambda=\mathcal E_\lambda^{m-1}$ and that $E$ and $F$ are as in the statement of this lemma. 
Consider an enlargement of the set $F$, defined by
$$\widetilde{F}=\bigcup_{y\in F}\Delta_{d/2}(y),$$
where $\Delta_{d/2}(y)$ are parabolic balls of radius $d/2$ centered at $y$. It follows that \(\mathrm{dist}(E,\widetilde F)>d/2\).
Then 

$$\int_F |\mathcal{E}^m_\lambda f|^2 dxdt\lesssim
\int_F |\mathcal{E}_\lambda\chi_{\widetilde F}\mathcal{E}^{m-1}_\lambda f|^2 dxdt
+\int_F |\mathcal{E}_\lambda\chi_{\R^{n+1}\setminus \widetilde F}\mathcal{E}^{m-1}_\lambda f|^2 dxdt
$$
$$\lesssim \int_{\widetilde{F}} |\mathcal{E}^{m-1}_\lambda f|^2 dxdt+
e^{-c\frac{d}{2\lambda}}\int_{\R^{n+1}\setminus \widetilde{F}}|\mathcal{E}^{m-1}_\lambda f|^2dxdt.
$$
Here we have used uniform $L^2$ boundedness of $\mathcal{E}_\lambda$ shown in \eqref{E-bound} for the first term
and the Lemma \ref{lemma:OffDiagonalResolvent} for the second term. 
Now we see that we can use the induction assumption for $m-1$ to handle the first term and repeated use of  
$L^2$ boundedness of $\mathcal{E}_\lambda$ for the second one. This yields to 
$$\int_F |\mathcal{E}^m_\lambda f|^2 dxdt\lesssim e^{-c\frac{d}{2\lambda}}\int_{E}|f|^2dxdt+e^{-c\frac{d}{2\lambda}}\int_{\R^{n+1}}|f|^2dxdt=2e^{-c\frac{d}{2\lambda}}\int_{E}|f|^2dxdt,
$$
since $f$ is supported on $E$.
\end{proof}

\section{Completing the estimate of \eqref{list}}

It remains to prove the bound stated in point 4) of \eqref{list}, which we do in this section.

\subsection{The averaging operator $\mathfrak{A}_\lambda$}
We start by revisiting \eqref{def-frak} where we have defined  the operator $\mathfrak{A}_\lambda$.

\noindent\textbf{Motivation.} The averaging operator $\mathfrak{A}_\lambda$ is introduced in order to extract a cancellation mechanism from the commutator
$\tilde{\mathcal P}^*_\lambda \partial_j(a_{n+1,j}\,\cdot) - (\tilde{\mathcal P}^*_\lambda \partial_j a_{n+1,j})\cdot$
appearing inside item~4) of \eqref{list}. Applied to a function $v$ that solves the adjoint equation, $\mathfrak{A}_\lambda v$ behaves, at scale~$\lambda$, like a slowly-varying ``frozen'' version of $v$, and Lemma~\ref{lemma:IminusPBoundedByS(v)} below quantifies precisely the extent to which $v - \mathfrak{A}_\lambda v$ is controlled by $S(v)$. The commutator then splits into (i) a piece in which the coefficient oscillation is paired against $(I-\mathfrak{A}_\lambda)v$ (controlled by the square function of $v$), and (ii) a piece in which $\mathfrak{A}_\lambda v$ is treated as essentially constant on $\lambda$-scale cubes (controlled via a Carleson-measure bound on the singular kernel). The temporal averaging estimate of Lemma~\ref{l8.5} below is the analytic input needed for the time-component of step (i): it is precisely the estimate that makes the temporal contribution to $v-\mathfrak{A}_\lambda v$ summable in the dyadic decomposition used later.

\begin{lemma}\label{lemma:IminusPBoundedByS(v)}
    For any function \(v\in W^{1,2}_{loc}(\R^{n+1}_+\times\R)\) with \(\mathcal{H}^*v=\divg(\vec{h})\) and \(1\leq p <\infty\) we have
    \[\Vert\mathcal{A}[(I-\mathfrak{A}_{\lambda})v]\Vert_{L^p(\mathbb{R}^{n+1})}\lesssim \Vert S(v)\Vert_{L^p(\mathbb{R}^{n+1})} + \Vert\mathcal{A}[\lambda \vec{h}]\Vert_{L^p(\mathbb{R}^{n+1})}.\]
\end{lemma}

\begin{proof} 
Fix a boundary point $(z,\tau)$ and consider its nontangential cone
$\Gamma(z,\tau)$. Recall the definition of the set $\Gamma^\lambda(z,\tau)$ by \eqref{eq-GL} - the slice at level 
$\lambda$ of the cone $\Gamma(z,\tau)$.

Without loss of generality, we can assume that the sets 
$\Gamma^\lambda(z,\tau)$ are parabolic cubes in the variables $(y,s)$ - if not we may enlarge the cones $\Gamma$ to make it so. This will just make the area function $\mathcal A$ on the enlarged cone bigger.
Then
\begin{equation}\label{eq8.22}
\mathcal{A}[(I-\mathfrak{A}_{\lambda})v](z,\tau)=\left(\int_0^\infty\lambda^{-1}\fint_{\Gamma^\lambda(z,\tau)}
|v-\mathfrak{A}_{\lambda}v|^2dyds\,d\lambda\right)^{1/2}.
\end{equation}
Since $\mathfrak{A}_{\lambda}v(y,\lambda,s)=\fint_{Q_\lambda(y,s)} v(y',\lambda,s')\,dy'\,ds'$,
it would be useful to know how this average compares to the average $\langle v\rangle_{\Gamma^\lambda(z,\tau)}$ of $v$ on the set $\Gamma^\lambda(z,\tau)$. We focus first on the difference
\begin{equation}\nonumber
\mathcal{A}(v-\langle v\rangle_{\Gamma^\lambda(z,\tau)})(z,\tau)=\left(\int_0^\infty\lambda^{-1}\fint_{\Gamma^\lambda(z,\tau)}
|v-\langle v\rangle_{\Gamma^\lambda(z,\tau)}|^2dyds\,d\lambda\right)^{1/2}.
\end{equation}
It is useful to integrate over the set of all points \((x,\tilde{\lambda},t)\), where \((x,t)\in \Gamma^\lambda(z,\tau)\) and \(\tilde{\lambda}\in (\lambda/2,\lambda)\), to make the interior integral solid using the Fubini's theorem. Abusing the notation a bit (since $\tilde\lambda$ is not the last variable if the vector \((x,\tilde{\lambda},t)\)), we denote this set as
$\Gamma^\lambda(z,\tau)\times (\lambda/2,\lambda)$. This will only make the bound bigger, that is
\begin{equation}\nonumber
\mathcal{A}(v-\langle v\rangle_{\Gamma^\lambda(z,\tau)})(z,\tau)\lesssim \left(\int_0^\infty\lambda^{-1}\fiint_{\Gamma^\lambda(z,\tau)\times (\lambda/2,\lambda)}
|v-\langle v\rangle_{\Gamma^{\tilde\lambda}(z,\tau)}|^2d\tilde\lambda dyds\,d\lambda\right)^{1/2}.
\end{equation}
Since for different values of $\tilde\lambda$ the averages are related via Poincar\'e's inequality in the $\lambda$-variable, this then leads to the estimate,
\begin{multline}\label{eq8.3}
\mathcal{A}(v-\langle v\rangle_{\Gamma^\lambda(z,\tau)})(z,\tau)\lesssim \left(\int_0^\infty\lambda^{-1}\fiint_{\Gamma^\lambda(z,\tau)\times (\lambda/2,\lambda)}
|v-\langle v\rangle_{\Gamma^\lambda(z,\tau)\times (\lambda/2,\lambda)}|^2d\tilde\lambda dyds\,d\lambda\right)^{1/2}
\\+\left(\int_0^\infty\lambda^{-1}\fiint_{\Gamma^\lambda(z,\tau)\times (\lambda/2,\lambda)}\lambda^2|\partial_\lambda v|^2d\tilde\lambda dyds\,d\lambda
\right)^{1/2}.
\end{multline}
Clearly the second term here is bounded by $S(v)(z,\tau)$.

To slightly shorten the notation, for a fixed time $s$ we denote by
$\tilde{v}^\lambda(s)$ the averages of $v$ over the level set where $s$ is constant:
$$\Gamma^\lambda(z,\tau)\times (\lambda/2,\lambda)\cap \{(y,\tilde{\lambda},s):\,(y,\tilde{\lambda})\in\R^{n+1}_+ \}.$$
If, in the first term on the righthand side of \eqref{eq8.3} the quantity
$\langle v\rangle_{\Gamma^\lambda(z,\tau)\times (\lambda/2,\lambda)}$ is replaced by 
$\tilde{v}^\lambda(s)$, then since $s$ is fixed we may use Poincare's inequality in the spatial variables to get 
 another term bounded by  $S(v)(z,\tau)$. Hence 
\begin{multline}\label{eq8.4}
\mathcal{A}(v-\langle v\rangle_{\Gamma^\lambda(z,\tau)})(z,\tau)\lesssim 
S(v)(z,\tau)\\+\left(\int_0^\infty\lambda^{-1}\fiint_{\Gamma^\lambda(z,\tau)\times (\lambda/2,\lambda)}
|\tilde{v}^\lambda(s)-\langle v\rangle_{\Gamma^\lambda(z,\tau)\times (\lambda/2,\lambda)}|^2d\tilde\lambda dyds\,d\lambda\right)^{1/2}.
\end{multline}

To handle the second term we shall prove the following lemma:

\begin{lemma}\label{l8.5} Let $Q$ be an interior parabolic cube such that $3Q\subset \R^{n+1}_+\times \R$.
For simplicity assume that we can write $Q$ as $\mathcal B\times (t_0-r^2,t_0+r^2)$.
We introduce the following averages in time using a smooth cutoff function. 
Let $\eta$ be a smooth nonnegative cutoff
function in spatial variables, supported in $2\mathcal B$ and satisfying $\eta=1$ on $\mathcal B$, with $|\nabla\eta|\lesssim r^{-1}$. For such $\eta$ and some $s\in\R$ consider the averages
$$\langle{v}_{\eta}\rangle(s)=c(\eta)\iint_{2\mathcal B}v(X,s)\eta(X)\,dX,$$
where $c(\eta)=(\int_{2\mathcal B}\eta)^{-1}\sim r^{-n-1}$.
Then we have
\begin{equation}\label{eq-zx}
\sup_{s,t\in (t_0-r^2,t_0+r^2)}|\langle{v}_{\eta}\rangle(s)-\langle{v}_{\eta}\rangle(t)|\lesssim r\left(\fiint_{2Q}\left[|\nabla v|^2+|\vec h|^2\right]\right)^{1/2}.
\end{equation}
\end{lemma}
With this in hand for $Q=\Gamma^\lambda(z,\tau)\times (\lambda/2,\lambda)$
 the last term of \eqref{eq8.4}  is handled as follows. We write
 \begin{multline}\label{eq-3}
|\tilde{v}^\lambda(s)-\langle v\rangle_{\Gamma^\lambda(z,\tau)\times (\lambda/2,\lambda)}|^2\lesssim
|\tilde{v}^\lambda(s)-\langle{v}_{\eta}\rangle(s)|^2+|\langle{v}_{\eta}\rangle(s)-\langle\langle{v}_{\eta}\rangle\rangle_s|^2\\+|\langle\langle{v}_{\eta}\rangle\rangle_s-
\langle v\rangle_{\Gamma^\lambda(z,\tau)\times (\lambda/2,\lambda)}|^2,
\end{multline}
where we used the shortcut $\langle\langle{v}_{\eta}\rangle\rangle_s$ to denote the average in time of $\langle{v}_{\eta}\rangle(s)$.
The first and last term of the expression above are handled by the Poincar\'e inequality in spatial variables only as it is easy to see that
$$|\tilde{v}^\lambda(s)-\langle {v_\eta}\rangle(s)|\lesssim r \fiint_{2\mathcal B}|\nabla v(X,s)|dX.$$ 
Therefore, squaring and integrating in $s$ gives the quantity
$$ r^2 \fiint_{2Q}|\nabla v|^2dXd\tau.$$
For the middle term of \eqref{eq-3} we use the claim of the lemma to obtain the square of the righthand side of \eqref{eq-zx}. Putting all terms together, we get that
\begin{multline}\label{eq8.6}
\mbox{The last term of \eqref{eq8.4}}\\\lesssim 
S_b(v)(z,\tau)+\left(\int_0^\infty\lambda^{-1}\fiint_{2\Gamma^\lambda(z,\tau)\times (\lambda/4,2\lambda)}
|\lambda \vec{h}|^2 d\tilde\lambda dyds\,d\lambda\right)^{1/2}
\\
\lesssim S_b(v)(z,\tau)+\left(\int_{\Gamma_{b}(z,\tau)}|\lambda \vec{h}|^2\lambda^{-n-3}d\lambda\,dx\,dt\right)^{1/2}
=S_b(v)(z,\tau)+\mathcal A_b(\lambda \vec{h})(z,\tau).
\end{multline}
Here the cone ${\Gamma_{b}(z,\tau)}$ has an enlarged aperture.

To estimate \eqref{eq8.22} we still need to deal with the difference
$|\mathfrak{A}_\lambda v - \langle v\rangle_{\Gamma^\lambda(z,\tau)}|$.
We now have all the estimates needed to do it. Again we bring in a solid integral term to have

\begin{multline}\label{eq8.22a}
\mathcal{A}( \langle v\rangle_{\Gamma^\lambda(z,\tau)}-\mathfrak{A}_{\lambda}v)(z,\tau)\\\lesssim\left(\int_0^\infty\lambda^{-1}\fiint_{\Gamma^\lambda(z,\tau)\times (\lambda/2,\lambda)}
|\langle v\rangle_{Q_{\tilde\lambda}(y,s)}-\langle v\rangle_{\Gamma^{\tilde\lambda}(z,\tau)}|^2d\tilde\lambda dyds\,d\lambda\right)^{1/2}.
\end{multline}
As before, inserting the averages over spatial cubes yields a term very similar to the last term of \eqref{eq8.3}. Therefore now we need to understand the difference of averages
$$|\langle v\rangle_{Q_{\lambda}(y,s)\times(\lambda/2,\lambda)}-\langle v\rangle_{\Gamma^{\lambda}(z,\tau)\times(\lambda/2,\lambda)}|^2.$$

Again it is necessary to bring in the smooth version of averages involving $\eta$. Modulo the quantity $S_b(v)(z,\tau)$, 
it suffices to consider
$$|\langle v_\eta\rangle_{Q_{\lambda}(y,s)\times(\lambda/2,\lambda)}-\langle v_\eta\rangle_{\Gamma^{\lambda}(z,\tau)\times(\lambda/2,\lambda)}|^2,$$
where $\eta$ is correspondingly chosen smooth cutoff as in Lemma \ref{l8.5}. We may use the same function $\eta$ for both terms as the union of Whitney cubes $Q_{\lambda}(y,s)\times(\lambda/2,\lambda)$ and 
$\Gamma^{\lambda}(z,\tau)\times(\lambda/2,\lambda)$ will be contained in a slightly larger 
Whitney cube. Then we apply  Lemma \ref{l8.5}, again obtaining terms similar to those in \eqref{eq8.6}.
Hence we have established Lemma \ref{lemma:IminusPBoundedByS(v)} assuming  Lemma \ref{l8.5} which we shall prove below.
\end{proof}

\begin{corollary}\label{cor-mathfrak} For $v$ chosen as in \eqref{e3}, and for any $1<p'<\infty$ for which the $L^{p'}$ Dirichlet problem for the adjoint operator $\mathcal H^*$ is solvable, the following bound holds:
\[\Vert\mathcal{A}[(I-\mathfrak{A}_{\lambda})v]\Vert_{L^{p'}(\mathbb{R}^{n+1})}\lesssim 1.\]
\end{corollary}

\begin{proof}
As noted earlier $\|S(v)\|_{L^{p'}}\lesssim 1$ and hence by Lemma \ref{lemma:IminusPBoundedByS(v)}
it remains to bound $\Vert\mathcal{A}(\lambda \vec{h})\Vert_{L^{p'}}$.
Note that the term \(\Vert\mathcal{A}(\lambda \vec{h})\Vert_{L^{p'}}^{p'}\) can be rewritten as 
    \begin{multline*}
        \Vert\mathcal{A}(\lambda \vec{h})\Vert_{L^{p'}}^{p'}=\int_{\mathbb{R}^{n+1}}\Big(\iint_{\Gamma(x,t)}|\lambda \vec{h}|^2 \lambda^{-n-3}dyd\lambda ds\Big)^{p'/2}dxdt
        \\
        \leq \int_{\mathbb{R}^{n+1}}\Big(\sum_{k\in \mathbb{Z}}\iint_{Q_{2^{k-1}}(x,2^k,t)}|\lambda \vec{h}|^2 \lambda^{-n-3}dyd\lambda ds\Big)^{p'/2}dxdt
        \\
        \leq \int_{\mathbb{R}^{n+1}}\left(\sum_{k\in \mathbb{Z}}2^k\left(\fiint_{Q_{2^{k-1}}(x,2^k,t)}|\vec h|^2dyd\lambda ds\right)^{1/2}\right)^{p'}dxdt
        \\
        \leq \int_{\mathbb{R}^{n+1}}T_1(\vec h)^{p'}dxdt=\Vert T_1(\vec h)\Vert_{L^{p'}}^{p'}.       
    \end{multline*}
Here in the second inequality we have used the fact that $\ell^1(\N)\subset\ell^2(\N)$  with appropriate norm estimate. The operator $T_1(\vec h)$, defined as in Definition~4.2.7 of \cite{U}, can for our purposes be described concretely as
\begin{equation}\label{eq:T1def}
T_1(\vec h)(x,t) := \sum_{k\in\Z} 2^k\left(\fiint_{Q_{2^{k-1}}(x,2^k,t)}|\vec h|^2 \,dy\, d\lambda\, ds\right)^{1/2},
\end{equation}
i.e., a discrete Littlewood--Paley square-function-type sum over dyadic scales of the $L^2$-averages of $\vec h$ on Whitney cubes at the corresponding height above $(x,t)$. 

We now make use of a tent space bound, namely that $\|T_1(\vec h)\|_{L^{p'}}\lesssim 1$ for all $p'>1$, for the $\vec h$ arising as the test field from Lemma~\ref{l1bb}. (See the remarks following the statement of Lemma~\ref{l1bb} earlier.) This bound is a real variable estimate proven by duality exactly as in Lemma 2.13 of \cite{KP2} (and in a more general form in \cite{MPT}), and is carried out explicitly for the parabolic setting in 
Lemma~4.2.12 of \cite{U}. That is, we obtain:
    \[\Vert\mathcal{A}(\lambda \vec{h})\Vert_{L^{p'}}\leq C(n,\lambda,p').\]
\end{proof}

Next we prove Lemma \ref{l8.5}. Calculating as in \cite{Din23}, we multiply the equation $\LL^* v=\divg(\vec{h})$ by $\eta$ and integrate it over the region
$\mathbb R^{n+1}\times [s,t]$, extending the function $v\eta$ by zero outside the support of $\eta$.
Because $\eta$ is a  
function of the spatial variables only, we have
$$\iint_{\mathbb R^{n+1}_+\times\{t\}}v\eta\,dY-\iint_{\mathbb R^{n+1}_+\times\{s\}}v\eta\,dY=
\iint_{\mathbb R^{n+1}_+\times[s,t]}(\partial_\tau v)\eta\,dY\,d\tau$$
$$=-\iint_{\mathbb R^{n+1}_+\times[s,t]}\mbox{\rm div}(A\nabla v)\eta\,dY\,d\tau
-\iint_{\mathbb R^{n+1}_+\times[s,t]}\mbox{\rm div}(\vec h)\eta\,dY\,d\tau$$
$$
=\iint_{\mathbb R^{n+1}_+\times[s,t]}(A\nabla v)\nabla\eta\,dY\,d\tau+\iint_{\mathbb R^{n+1}_+\times[s,t]}\vec h\cdot\nabla\eta\,dY\,d\tau.
$$
 The term $\nabla\eta$ in the right-hand side integral is bounded by
 $r^{-1}$, and the length of the interval $[s,t]$  is bounded by $r^2$, therefore
$$\left|\iint_{\mathbb R^{n+1}_+\times\{t\}}v\eta\,dY-\iint_{\mathbb R^{n+1}_+\times\{s\}}v\eta\,dY\right|\lesssim
r^{n+2}\fint_{t_0-r^2}^{t_0+r^2}\,\fiint_{2\mathcal B}[|\nabla v|+|\vec h|]\,dX\,d\tau.$$
We multiply by $c(\eta)$ and take sup over all $s$ and $t$. This gives
\begin{multline*}
\sup_{s,t\in(t_0-r^2,t_0+r^2)}|\langle{v}_{\eta}\rangle(t)-\langle{v}_{\eta}\rangle(s)|\\\lesssim r \fiint_{2Q}[|\nabla v|+|\vec h|]dXd\tau
\le r\left(\fiint_{2Q}[|\nabla v|^2+|\vec h|^2]dXd\tau\right)^{1/2}. 
\end{multline*}
\qed

\subsection{The area bound of the difference}

\begin{lemma}\label{lemma:AreaFctOfDiff}
 For any function \(v\in W^{1,2}_{loc}(\R^{n+1}_+\times\R)\) solving \(\mathcal{H}^*v=\mathrm{div}(\vec h)\) we have for all $1\le p<\infty$
\[\Vert \mathcal{A}\Big(\lambda\big(\tilde{\mathcal{P}}^*_{\lambda}\partial_j(a_{n+1,j}\mathfrak{A}_{\lambda}) - (\tilde{\mathcal{P}}^*_{\lambda}\partial_ja_{n+1,j})\mathfrak{A}_\lambda\big)v\Big)\Vert_{L^{p}}\lesssim \Vert N(v)\Vert_{L^{p}}+\Vert S(v)\Vert_{L^{p}}+\Vert \mathcal{A}(\lambda \vec{h})\Vert_{L^p},\]
Here \(\tilde{\mathcal{P}}_\lambda=\mathcal{P}_{\lambda,k}\) for an integer $k\sim m/2$ and \(\mathfrak{A}_\lambda\) is the averaging operator introduced in \eqref{def-frak}.
\end{lemma}

Using the bounds for $v$ it follows that:

\begin{corollary}\label{c712} For all $1<p'<\infty$ for which the $L^{p'}$ Dirichlet problem for the adjoint PDE  
$\mathcal{H}^*u=0$ on $\R^{n+1}_+\times\R$ is solvable we have that 
\[\Vert \mathcal{A}\Big(\lambda\big(\tilde{\mathcal{P}}^*_{\lambda}\partial_j(a_{n+1,j}\mathfrak{A}_{\lambda}) - (\tilde{\mathcal{P}}^*_{\lambda}\partial_ja_{n+1,j})\mathfrak{A}_\lambda\big)v\Big)\Vert_{L^{p'}}\lesssim 1.\]
\end{corollary}

Hence the claim 4) of the \eqref{list} holds. We now prove Lemma \ref{lemma:AreaFctOfDiff}.

\begin{proof}
Fix a boundary point $(z,\tau)\in\R^n\times \R$ and let $\Gamma(z,\tau)$ be its nontangential cone.
Consider a point $(X,t)=(x,\lambda,t)\in \Gamma(z,\tau)$. Let us also recall the notation introduced previously,  $\Gamma^\lambda(z,\tau)$ is the intersection of $\Gamma(z,\tau)$ with the hyperplane $x_{n+1}=\lambda$.

We can then write 
\begin{multline*}
\lambda\big[\tilde{\mathcal{P}}^*_{\lambda}\partial_j(a_{n+1,j}\mathfrak{A}_{\lambda}v) - (\tilde{\mathcal{P}}^*_{\lambda}\partial_ja_{n+1,j})\mathfrak{A}_\lambda v\big](x,\lambda,t)=
\\
\lambda\big[\tilde{\mathcal{P}}^*_{\lambda}\partial_j(a_{n+1,j}\mathfrak{A}_{\lambda}v) - (\tilde{\mathcal{P}}^*_{\lambda}\partial_ja_{n+1,j})\langle v\rangle_{Q_\lambda(z,\tau)}\big](x,\lambda,t)
\\
+\lambda(\tilde{\mathcal{P}}^*_{\lambda}\partial_ja_{n+1,j})[\langle v(\cdot,\lambda,\cdot)\rangle_{Q_\lambda(z,\tau)}-\langle v(\cdot,\lambda,\cdot)\rangle_{Q_\lambda(x,t)}](x,\lambda,t)=:I_\lambda+II_\lambda.
\end{multline*}

The second term is easy to handle. For the area function arising from it:
\begin{multline*}
\iint_{\Gamma(z,\tau)}\lambda^{-n-1}|\tilde{\mathcal{P}}^*_{\lambda}\partial_ja_{n+1,j}|^2|\langle v(\cdot,\lambda,\cdot)\rangle_{Q_\lambda(x,t)}-\langle v(\cdot,\lambda,\cdot)\rangle_{Q_\lambda(z,\tau)}|^2 dX dt\\
\le N(\langle v(\cdot,\lambda,\cdot)\rangle_{Q_\lambda(x,t)}-\langle v(\cdot,\lambda,\cdot)\rangle_{Q_\lambda(z,\tau)})^2(z,\tau)\iint_{\Gamma(z,\tau)}\lambda^{-n-1}|\tilde{\mathcal{P}}^*_{\lambda}\partial_ja_{n+1,j}|^2dX dt,
\end{multline*}
and since $N(\langle v(\cdot,\lambda,\cdot)\rangle_{Q_\lambda(x,t)}-\langle v(\cdot,\lambda,\cdot)\rangle_{Q_\lambda(z,\tau)})(z,\tau)\le 2N_b(v)(z,\tau)$ for $N_b$ defined with a cone of wider aperture, it follows that
$$\mathcal {A}(II_\lambda)(z,\tau)\lesssim N_b(v)(z,\tau)\left(\iint_{\Gamma(z,\tau)}\lambda^{-n-3}|\lambda\tilde{\mathcal{P}}^*_{\lambda}\partial_ja_{n+1,j}|^2dXdt \right)^{1/2}.$$
Hence, using the Carleson bound \eqref{eq:CarlesonFctBound} and the classical stopping time argument we get that for all $1\le p<\infty$:
$$\|\mathcal {A}(II_\lambda)\|_{L^p}\lesssim \|C(\lambda\tilde{\mathcal{P}}^*_{\lambda}\partial_ja_{n+1,j} )\|_{L^\infty}\|N(v)\|_{L^p}\lesssim \|N(v)\|_{L^p}.$$

Considering the term $I_\lambda$, we see that, for a fixed $\lambda$, 
$\tilde{v}(\lambda):=\langle v\rangle_{Q_\lambda(z,\tau)}=\langle v(\cdot,\lambda,\cdot)\rangle _{Q_\lambda(z,\tau)}$ 
is constant for all $(x,\lambda,t)\in\Gamma(z,\tau)$ and so the two 
terms can be combined. That is,
$$I_\lambda=\lambda\big[\tilde{\mathcal{P}}^*_{\lambda}\partial_j(a_{n+1,j}(\mathfrak{A}_{\lambda}v-\tilde{v}(\lambda)))](x,\lambda,t).$$

Fix $\lambda>0$ and consider the decomposition of the boundary $\R^n\times\R$ as in \eqref{eq6.2a}
such that the property \eqref{eq6.2} holds for all $j>j_0$. To avoid confusion (as the line above contains an unrelated index $j$) we write below $\partial_k$ instead. Then
\begin{multline*}
I_\lambda= \lambda\big[\tilde{\mathcal{P}}^*_{\lambda}\partial_k[\chi_{\Delta_{2^{j_0}\lambda}}a_{n+1,k}(\mathfrak{A}_{\lambda}v-\tilde{v}(\lambda))]]\\
+
\sum_{j>j_0}\lambda\big[\tilde{\mathcal{P}}^*_{\lambda}\partial_k[\chi_{\left(\Delta_{2^j\lambda}\setminus \Delta_{2^{j-1}\lambda}\right)}a_{n+1,k}(\mathfrak{A}_{\lambda}v-\tilde{v}(\lambda))]]=III_\lambda+IV_\lambda.
\end{multline*}
For the term $IV$ using the off-diagonal estimate (2) of Lemma \ref{lemma:OffDiagonalResolvent2} 
and Minkowski's inequality
we have that
\begin{multline}\label{eq8.12}
\mathcal {A}_a(IV_\lambda)^2(z,\tau)\\=\int_0^\infty \lambda^{-n-3}\int_{\Gamma_a^\lambda(z,\tau)}
\left|\sum_{j>j_0}\lambda\tilde{\mathcal{P}}^*_{\lambda}\partial_k[\chi_{\left(\Delta_{2^j\lambda}\setminus \Delta_{2^{j-1}\lambda}\right)}a_{n+1,k}(\mathfrak{A}_{\lambda}v-\tilde{v}(\lambda))]
\right|^2dxdt d\lambda\\
\le\int_0^\infty\lambda^{-n-3}\left(\sum_{j>j_0}\left(\int_{\Gamma_a^\lambda(z,\tau)}
\left|\lambda\tilde{\mathcal{P}}^*_{\lambda}\partial_k[\chi_{\left(\Delta_{2^j\lambda}\setminus \Delta_{2^{j-1}\lambda}\right)}a_{n+1,k}(\mathfrak{A}_{\lambda}v-\tilde{v}(\lambda))]\right|^2 dxdt
\right)^{1/2}\right)^2d\lambda\\
\lesssim \int_0^\infty\lambda^{-n-3}\left(\sum_{j>j_0}e^{-c2^j}\left(\int_{\Delta_{2^j\lambda}(z,\tau)}
 |\mathfrak{A}_{\lambda}v-\tilde{v}(\lambda)|^2 dx dt
\right)^{1/2}\right)^{2}d\lambda\\
\lesssim \sum_{j>j_0}\int_0^\infty\lambda^{-n-3} e^{-c2^j} \int_{\Delta_{2^j\lambda}(z,\tau)}
 |\mathfrak{A}_{\lambda}v-\tilde{v}(\lambda)|^2 dx dt\,d\lambda.
\end{multline}
using H\"older's inequality to obtain the last line. 

We consider the difference in the last line. $\mathfrak{A}_{\lambda}v(x,t)=\langle v(\cdot,\lambda,\cdot)\rangle_{Q_\lambda(x,t)}$ and $\tilde{v}(\lambda)=\langle v(\cdot,\lambda,\cdot)\rangle_{Q_\lambda(z,\tau)}$
are two averages over cubes of the same size (just shifted). It would be more convenient if these two averages
were over all spatial variables. By Fubini, for any such $Q$ as above,
\begin{equation}\label{eq8.13}
|\langle v(\cdot,\lambda,\cdot)\rangle_{Q}-\langle v\rangle_{Q\times(\lambda,2\lambda)}|^2
\lesssim \lambda^2\left(\fiint_{Q\times(\lambda,2\lambda)}|\partial_\lambda v|dX dt\right)^2
\lesssim \lambda^2\fiint_{Q\times(\lambda,2\lambda)}|\partial_\lambda v|^2dX dt,
\end{equation}
and hence 
\begin{multline}\label{eq8.14}
\int_{\Delta_{2^j\lambda}(z,\tau)}
 |\mathfrak{A}_{\lambda}v-\tilde{v}(\lambda)|^2 dx dt\\\lesssim
 \int_{\Delta_{2^j\lambda}(z,\tau)} |\langle v\rangle_{Q_\lambda(x,t)\times(\lambda,2\lambda)}-\langle v\rangle_{Q_\lambda(z,\tau)\times(\lambda,2\lambda)}|^2 dx dt
 + \lambda^{n+4}\fiint_{\Delta_{2^{j+1}\lambda}(z,\tau)\times(\lambda,2\lambda)}|\partial_\lambda v|^2 dX dt.
\end{multline}
To understand the first term of the last line we need to estimate how the averages of two cubes of size $\lambda$ shifted in space and time will differ when their (parabolic) distance is at most $2^j\lambda$. We do not have to be very precise with our estimate, as the term $e^{-c2^j}$ decays extremely rapidly.
With the help of an intermediate cube $Q_\lambda(x,\tau)\times(\lambda,2\lambda)$, the matter can be
reduced to two fundamental cases. 

The first case is where two cubes are only shifted in space but not in time. Then a calculation similar to
\eqref{eq8.13} yields that 
$$ |\langle v\rangle_{Q_\lambda(x,\tau)\times(\lambda,2\lambda)}-\langle v\rangle_{Q_\lambda(z,\tau)\times(\lambda,2\lambda)}|^2\lesssim (2^j\lambda)^2\fiint_{H}|\nabla_x v|^2dX dt,
$$
where $H\subset Q_{2^{j+1}\lambda}(z,\tau)\times(\lambda,2\lambda)$ is the convex hull of the union of the two cubes. The second case is where the two cubes have the same spatial coordinates but differ in time.
In that case, we need to bring in the smooth averages from Lemma \ref{l8.5}. This lemma implies that

$$ |\langle v_\eta\rangle_{2Q_\lambda(x,\tau)\times(\lambda/2,3\lambda)}-\langle v_\eta\rangle_{2Q_\lambda(x,t)\times(\lambda/3,3\lambda)}|^2\lesssim (2^j\lambda)^2\fiint_{\tilde H}\left[|\nabla_{x,\lambda} v|^2+|\vec h|^2\right]dX dt,
$$
where $\tilde H$ is a small enlargement of the convex hull of the two cubes. Relating the two averages $\langle v\rangle$ and $\langle v_\eta\rangle$ yields another term of the type 
$\lambda^2\displaystyle\fiint_{\tilde 2Q_\lambda\times(\lambda/2,3\lambda)}|\nabla_{x,\lambda} v|^2dX dt$.
Hence we have
$$ |\langle v\rangle_{Q_\lambda(x,t)\times(\lambda,2\lambda)}-\langle v\rangle_{Q_\lambda(z,\tau)\times(\lambda,2\lambda)}|^2\lesssim 
(2^{j}\lambda)^2\fiint_{\Delta_{2^{j+1}\lambda}(z,\tau)\times(\lambda/2,3\lambda)}\left[|\nabla_{x,\lambda} v|^2+|\vec h|^2\right]dX dt.
$$
And, by \eqref{eq8.14},
$$
\int_{\Delta_{2^j\lambda}(z,\tau)}
 |\mathfrak{A}_{\lambda}v-\tilde{v}(\lambda)|^2 dx dt\lesssim
 (2^{j}\lambda)^{n+4}\fiint_{\Delta_{2^{j+1}\lambda}(z,\tau)\times(\lambda/2,3\lambda)}\left[|\nabla_{x,\lambda} v|^2+|\vec h|^2\right]dX dt.
$$
We use this in \eqref{eq8.12}. It follows that
\begin{multline*}
\mathcal {A}_a(IV_\lambda)^2(z,\tau)\\\lesssim
 \sum_{j>j_0}\int_0^\infty\lambda^{-n-3} e^{-c2^j} (2^{j}\lambda)^{n+4}\fiint_{\Delta_{2^{j+1}\lambda}(z,\tau)\times(\lambda/2,3\lambda)}\left[|\nabla_{x,\lambda} v|^2+|\vec h|^2\right]dX dt\\
 \le  \sum_{j>j_0}e^{-c2^j}2^{2j} \iint_{\Gamma_{2^{j+2}}(z,\tau)}
 \left[|\nabla_{x,\lambda} v|^2+|\vec h|^2\right]\lambda^{-n-1}dXdt.
\end{multline*}
Above, consistent with the previous notation,  $\Gamma_{2^{j+2}}(z,\tau)$ denotes the nontangential cone of aperture $2^{j+2}$. By Minkowski we conclude that
$$\mathcal {A}_a(IV_\lambda)(z,\tau)\lesssim \sum_{j\ge j_0}e^{-c2^{j-1}}2^j\left[ S_{2^{j+2}}(v)+
\mathcal A_{2^{j+2}}(\lambda \vec{h})\right](z,\tau).
$$
The subscripts indicate that $S$ and $\mathcal A$ are defined using enlarged cones.
In the $L^p$ norm, the area function of cones with
different apertures can be estimated
via real variable arguments; namely, $\|\mathcal A_{2^{j}}(\cdot)\|_{L^p}\le e^{c(n,p)j}\|\mathcal A_{1}(\cdot)\|_{L^p}$.
Therefore,
$$\|\mathcal {A}_a(IV_\lambda)\|_{L^p}\lesssim \sum_{j\ge j_0}2^je^{-c2^{j-1}+c(n,p)j}\left[ \|S_1(v)\|_{L^p}+
\|\mathcal A_{1}(\lambda \vec{h})\|_{L^p}\right]\lesssim \|S(v)\|_{L^p}+\|\mathcal A(\lambda \vec{h})\|_{L^p},
$$
using summability of the series $\sum_{j\ge j_0}2^je^{-c2^{j-1}+c(n,p)j}$,
which takes care of the term $IV$ completely. 

We now deal with the term $III_\lambda$, which is straightforward since we have the bound
$$\int_{\R^n\times\R}|\lambda \nabla_\parallel \mathcal P_\lambda f|^2 dx dt\lesssim 
\int_{\R^n\times\R}|f|^2dxdt,$$
by iterating  \eqref{E-bound}. Then, by duality,
$$\int_{\R^n\times\R}|\lambda  \mathcal P^*_\lambda \divg_\parallel(g)|^2 dx dt\lesssim 
\int_{\R^n\times\R}|g|^2dxdt,$$
and in particular for the term $III_\lambda$
\begin{multline*}
\int_{\Gamma^\lambda_a(z,\tau)}\left|\lambda\tilde{\mathcal{P}}^*_{\lambda}\partial_k[\chi_{\Delta_{2^{j_0}\lambda}}a_{n+1,k}(\mathfrak{A}_{\lambda}v-\tilde{v}(\lambda))]\right|^2dxdt\\
\lesssim \int_{\R^n\times\R}|\chi_{\Delta_{2^{j_0}\lambda}}a_{n+1,k}(\mathfrak{A}_{\lambda}v-\tilde{v}(\lambda))|^2dxdt\le \|A\|_{L^\infty}
\int_{\Delta_{2^{j_0}}}|\mathfrak{A}_{\lambda}v-\tilde{v}(\lambda)|^2dxdt.
\end{multline*}
From this we get 
$$
\mathcal {A}_a(III_\lambda)^2(z,\tau)
\lesssim \int_0^\infty\lambda^{-n-3} \int_{\Delta_{2^{j_0}\lambda}(z,\tau)}
 |\mathfrak{A}_{\lambda}v-\tilde{v}(\lambda)|^2 dx dt\,d\lambda.
$$
which is an estimate analogous to \eqref{eq8.12} but with no decay, which is not needed, as this is a single term. The rest of the argument is identical to that for term IV. This concludes the proof.

\end{proof}

\bibliographystyle{alpha}

\end{document}